\documentclass[a4paper,12pt]{amsart}
\usepackage[utf8]{inputenc}
\usepackage[english]{babel}
\usepackage[text={160mm,232mm},centering]{geometry}
\usepackage{amsmath,amssymb,amsthm,mathabx,verbatim}
\usepackage{tikz-cd}
\usepackage[bookmarks=true]{hyperref}
\usepackage{booktabs}
\usepackage{caption}
\usepackage{enumerate}
\usepackage{cleveref}
\usepackage{multirow}
\usepackage{array}
\newcolumntype{C}[1]{>{\centering\arraybackslash}m{#1}}

\numberwithin{equation}{section}

\newtheorem{thm}[equation]{Theorem} %teorema numerato
\newtheorem{prop}[equation]{Proposition} %proposizione numerata
\newtheorem{lem}[equation]{Lemma} %lemma numerato
\newtheorem{cor}[equation]{Corollary} %corollario numerato
\newtheoremstyle{named}{}{}{\itshape}{}{\bfseries}{.}{.5em}{#3}
\theoremstyle{named} 
\newtheorem*{namedtheorem}{Theorem} %teorema con nome [da scrivere dentro parentesi quadra] senza numero
\theoremstyle{remark}
\newtheorem{rem}[equation]{Remark} %remark numerato
\theoremstyle{definition}
\newtheorem{defi}[equation]{Definition} %definizione numerata

\numberwithin{table}{section}

\newcommand{\Xns}{X_{\textnormal{ns}}}
\newcommand{\Xs}{X_{\textnormal s}}

\newcommand{\ol}{\overline}
\newcommand{\tx}{\textnormal}
\newcommand{\Q}{\mathbb Q}
\newcommand{\Z}{\mathbb{Z}}
\newcommand{\C}{\mathbb{C}}
\newcommand{\R}{\mathbb R}
\newcommand{\F}{\mathbb F}
\newcommand{\HH}{\mathbb{H}}
\newcommand{\PP}{\mathbb P}
\newcommand{\calO}{\mathcal{O}}

\newcommand{\End}{\mathrm{End}}
\newcommand{\Gal}{\mathrm{Gal}}
\newcommand{\Aut}{\mathrm{Aut}}
\newcommand{\GL}{\mathrm{GL}_2}
\newcommand{\GLn}{\mathrm{GL}_2(\mathbb Z/n\mathbb Z)}

\newcommand{\SLpe}{\mathrm{SL}_2(\mathbb Z/p^e\mathbb Z)}

\newcommand{\ClZ}{\widetilde{{\mathrm{Cl}}}_{\mathfrak m}}
\newcommand{\new}{\textnormal{new}}
\newcommand{\ns}{\textnormal{ns}}
\newcommand{\s}{\textnormal{s}}
\newcommand{\CM}{\textnormal{C}}

\newcommand{\smt}[4]{\left( \begin{smallmatrix} #1 &#2\\ #3 &#4\\ \end{smallmatrix} \right) }
\newcommand{\lto}{\longrightarrow}

\begin{document}
\title{Automorphisms of Cartan modular curves of prime and composite level}

\author{Valerio Dose}
\email{vdose@luiss.it}
\address{Dipartimento di Economia e Finanza, LUISS "Guido Carli", Viale Romania 32, 00197 Roma, Italy}

\author{Guido Lido}
\email{guidomaria.lido@gmail.com}
\address{Universit\`a ``Tor Vergata'', Rome 00133, Italy}

\author{Pietro Mercuri}
\email{mercuri.ptr@gmail.com}
\address{Universit\`a ``Sapienza'', Rome 00161, Italy}

\thanks{The third author has been supported by the research grant ``Ing. Giorgio Schirillo'' of the Istituto Nazionale di Alta Matematica ``F. Severi'', Rome. We would also like to thank Claudio Stirpe, for giving us access to some of his source code which we used for this work, and the anonymous Referee, for the very careful reading and for the many useful comments}

%\subjclass[2020]{14G35,11G18,11G30,11G05,11G15}

\begin{abstract}
We study the automorphisms of modular curves associated to Cartan subgroups of $\GLn$ and certain subgroups of their normalizers. We prove that if $n$ is large enough, all the automorphisms are induced by the ramified covering of the complex upper half-plane. We get new results for non-split curves of prime level $p\ge 13$: the curve $X_\ns^+(p)$ has no non-trivial automorphisms, whereas the curve $X_\ns(p)$ has exactly one non-trivial automorphism. Moreover, as an immediate consequence of our results we compute the automorphism group of $X_0^*(n):=X_0(n)/W$, where $W$ is the group generated by the Atkin-Lehner involutions of $X_0(n)$ and $n$ is a large enough square.
\end{abstract}

\maketitle

\section*{Introduction}

Since the 1970s many efforts have been made to determine automorphisms of modular curves and in particular to establish whether a modular curve has other automorphisms besides the expected ones. Indeed, infinite automorphisms naturally arise when the curve has genus zero or one. Moreover, since the components of modular curves over $\C$ can be seen as compactification of quotients of the complex upper half-plane $\mathbb H$, some automorphisms of $\mathbb H$ induce automorphisms of the quotient modular curve. Such automorphisms are called \textit{modular} and their determination is a purely group theoretic problem.

The focus has been classically placed on the modular curves $X_0(n)$ associated to a Borel subgroup of $\GLn$ (e.g., upper triangular matrices), with $n$ a positive integer. For these curves, modular automorphisms played an important role in the development of the theory of modular curves. They were determined in the seminal paper \cite{AL}, with a small gap which was later filled in a couple of different ways (see \cite{ASNormGamma0}, \cite{BarsNormGamma0}). Meanwhile, a complete picture about the remaining automorphisms of $X_0(n)$ has been painted through the decades by the works \cite{OggPrime}, \cite{OggDe}, \cite{KM}, \cite{Elkies}, \cite{Harrison108}. Also some works in this century (e.g., \cite{BakerHas}, \cite{Mer0}, \cite{GonzSplit}) took on the case of the modular curves $X_0(p)/\langle w_p\rangle$ and $X_0(p^2)/\langle w_{p^2}\rangle$, where $w_p$ and $w_{p^2}$ are the Atkin-Lehner involutions of the respective modular curve.

More recently, great interest has been generated in modular curves associated to different subgroups of $\GLn$, in particular to normalizers of Cartan subgroups for $n=p$ prime. This is mainly due to the fact that rational points on these curves help classifying rational elliptic curves whose associated Galois representation modulo $p$ is not surjective. This is directly linked to a question formulated by Serre (also known as \textit{uniformity conjecture}) in the 1970s (\cite{Ser}). After the works \cite{MazurRI}, on the Borel case, and \cite{BPsplit}, \cite{BPR}, on the \textit{split} Cartan case, the only part of this problem left to understand nowadays is equivalent to asking whether, for almost every prime $p$, the modular curve $\Xns^+(p)$ associated to the normalizer of a \textit{non-split} Cartan subgroup of $\text{GL}_2(\mathbb Z/p\mathbb Z)$ has other rational points besides the expected ones, namely the CM points of class number one. This equivalence led to a certain amount of research driven towards computing equations and finding rational points of modular curves associated to non-split Cartan subgroups and their normalizers (see for example \cite{Baran9}, \cite{Baran13}, \cite{ExplCK}, \cite{DFGS}, \cite{DMS}, \cite{Merns}).

A curious connection between the problem of determining rational points and the one of determining automorphisms in a modular curve is given by the fact that in the case of the Borel modular curves $X_0(p)$ of genus at least 2, the sole occurrence of unexpected rational points ($p=37$) in the setting of Serre's uniformity conjecture, happens in the presence of an unexpected automorphism of the corresponding modular curve. A further connection is made in \cite{DoseAut}, where is proven that, for $p\ge29$, the absence of unexpected rational points of the curve $\Xns^+(p)$ implies the absence of unexpected rational automorphisms of the modular curve $\Xns(p)$ associated to a non-split Cartan subgroup of $\text{GL}_2(\mathbb Z/p\mathbb Z)$.

The first work centered on automorphisms of non-split Cartan modular curves has been \cite{DFGS}, in which the existence of an unexpected automorphism of $\Xns(11)$ is proven. Some partial results on the automorphisms of $\Xns(p)$ and $\Xns^+(p)$, for almost every prime $p$, were proven in \cite{DoseAut}, while in \cite{GonzConst} the full determination of the automorphism group is obtained for low primes ($p\le 31$). 

In the present work, we prove unconditionally that every automorphism of $\Xns(p)$ and $\Xns^+(p)$ is modular for $p\ge 13$. In fact, we also extend this to composite level $n$ where we can define Cartan subgroups of mixed split/non-split type. The scope of our study concerns Cartan subgroups and also a specific subgroup of their normalizer in $\GLn$ which we call \textit{Cartan-plus} subgroup. However, in most cases, for example when $n$ is odd, a Cartan-plus subgroup actually coincides with the normalizer of the relative Cartan subgroup. We prove the following result:

\begin{namedtheorem}[\Cref{thm:Cartan_aut}]
Let $n\ge 10^{400}$ be an integer and let $H<\GLn$ be either a Cartan or a Cartan-plus subgroup. Then every automorphism of $X_H$ is modular, hence we have
\begin{align*}
\Aut(X_H) &\cong \begin{cases}  N'/H' \times \Z/2\Z, & \tx{if } n\equiv 2 \bmod 4 \tx{ and $H$ is a Cartan-plus  split at }2, \\
N'/H', & \tx{otherwise,}
\end{cases}
\end{align*}
where $N'<\mathrm{SL}_2(\Z/n\Z)$ is the normalizer of $H':=H\cap\tx{SL}_2(\Z/n\Z)$.
\end{namedtheorem}
The huge bound of $10^{400}$ comes from \Cref{prop:field_Cartan(n)} hence from our estimates of the dimension of the CM part of $J_0(n)$. However, explicit computations can make the method work for low levels (see \Cref{tab:lowlevels} in the Appendix). It may be interesting to note that the modular curve associated to a Cartan-plus subgroup of $\GLn$ which is split at every prime dividing $n$ is isomorphic to the modular curve $X_0^*(n^2):=X_0(n^2)/W$, where $W$ is the group generated by Atkin-Lehner involutions of the Borel curve $X_0(n^2)$. For these curves, the case where $n$ is squarefree  has been recently determined in \cite{BarsGonzAut}.

In the case $n=p^e$, where $p$ is a prime number, we can refine the techniques developed and obtain a more complete result:

\begin{namedtheorem}[\Cref{thm:powerprimes}]
Let $p$ be a prime number and let $e$ be a positive integer. If $p^e>11$ and $p^e\notin\{3^3, 2^4, 2^5, 2^6\}$, then all the automorphisms of $X_\ns(p^e), X_\ns^+(p^e), X_\tx{s}(p^e)$ and $X_\tx{s}^+(p^e)$ are modular and
\[
\begin{aligned}
 \Aut(X_\ns(p^e)) &\cong \Z/2\Z, &
 \Aut(X_\ns^+(p^e)) &\cong \{1 \}, \\
 \Aut(X_\tx{s}(p^e)) &\cong \begin{cases}
(\Z/8\Z)^2 \rtimes (\Z/2\Z), & \tx{ if }p=2, \\
\Z/3\Z\times S_3, & \tx{ if }p=3,  \\
\Z/2\Z, & \tx{ if }p>3, \end{cases} &
 \Aut(X_\tx{s}^+(p^e)) &\cong \begin{cases}
\Z/8\Z, & \tx{ if }p=2, \\
\Z/3\Z, & \tx{ if }p=3,  \\
\{1\}, & \tx{ if }p>3,
\end{cases}
\end{aligned}
\]
where the above semidirect product $(\Z/8\Z)^2 \rtimes \Z/2\Z$ is described in Table \ref{table2}.
\end{namedtheorem}

\begin{namedtheorem}[\Cref{cor:aut_prime_level}]
Let $p \ge 13$ be a prime number. Then the group of automorphisms of $X_\ns^+(p)$ is trivial and the group of automorphisms of $X_\ns(p)$ has order $2$.
\end{namedtheorem}

The idea of the proof is the following. We start by showing that an automorphism $u$ of $X_H$ is defined over some explicit compositum of quadratic fields. This is done by studying the endomorphisms of the jacobian of $X_H$, which, through an isogeny relation, can be seen as endomorphisms of the well known modular jacobians $J_0(n)$. The fact that $u$ is defined over a compositum of quadratic fields, together with the Eichler-Shimura Relation, allow us to show that $u$ ``almost" commutes with the Hecke operators. We can then describe the action of Hecke operators on cusps and branching points of the cover $\mathbb H\rightarrow X_H$, and prove that $u$ preserves both these sets. This allows to lift $u$ to an automorphism of $\mathbb H$.

The main technical novelty of our proofs is the analysis of the action of Hecke operators mentioned above, which permits us to prove the result about automorphisms without exploiting and worrying about the field of definition of the cusps and CM points, which has been instead instrumental for determining automorphisms of modular curves throughout the literature in the past. In fact, both \cite{KM} and \cite{DoseAut} use the field of definition of the cusps to prove that an automorphism must preserve the set of cusps. In \cite{KM}, this is enough to exclude the existence of non-modular automorphisms, in combination with the rich action of the modular automorphisms of $X_0(n)$ on the set of cusps. In \cite{DoseAut}, the lack of the preservation result on the set  of branching points confines the analysis to the levels where there are no branching points at all. We also give à la Chen results to describe jacobians of Cartan modular curves in terms of jacobians of Borel modular curves and we give an explicit upper bound on the dimension of the CM part of the jacobian of Borel modular curves.

The structure of the paper is the following.

In Section \ref{Sec:ModularCurves} we define modular curves associated to general subgroups of $\GLn$ and we give an equivalent condition to the fact that a point of a modular curves branches in the covering of the curve by $\mathbb H$. 

In Section \ref{Sec:Hecke} we study the action of Hecke operators on modular curves. In particular we focus on the action on the cusps and the other points which could branch in the covering by $\mathbb H$. Such points are associated to elliptic curves with $j$-invariant equal to $0$ or $1728$. 

In Section \ref{Sec:Cartan} we define Cartan and Cartan-plus subgroups of $\GLn$ for every positive integer $n$. We also define the relative modular curves of composite level. Then we prove that the jacobian of a Cartan modular curve is a quotient of the jacobian of some Borel modular curve. When $n=p^e$, this is done applying the techniques of \cite{ChenPowers} and \cite{Edix} to a previously unexplored case, and for $n$ general we combine these results. We also extend the results of \cite{ChenPowers} to the case of even level. 

In Section \ref{Sec:FieldDef} we prove that all the automorphisms of Cartan modular curves must be defined on a compositum of quadratic fields when the level $n$ is large enough. To do this, we use a geometrical criterion that we can apply by bounding the dimension of the CM part of the jacobian of Cartan modular curves. This last step is obtained using the isogenies of Section \ref{Sec:Cartan} and computing explicit bounds for the CM part of the jacobians of Borel modular curves. Furthermore, we refine the results in the case $n=p^e$, with $p$ prime.

Finally, in Section \ref{Sec:Aut} we prove the results stated above about automorphisms. After showing that each automorphism must preserve the cusps and the set of branching points of the covering by $\mathbb H$, we prove that there are no non-modular automorphisms. Thus, we compute the modular automorphisms to complete the analysis and we discuss their field of definition. We first concentrate on Cartan modular curves of general level $n$. Then we adapt the strategy to the case $n=p^e$, with $p$ prime, giving the complete result for $\Xns(p)$ and $\Xns^+(p)$, and improving the result we obtained for the general level in the cases of $\Xs^+(p^e)$, $\Xns(p^e)$ and $\Xns^+(p^e)$. To treat some of the small level cases, we use the criterion of \cite{GonzConst} and some ad hoc arguments which we verify through an algorithm implemented in MAGMA (\cite{MAGMA}) which is available at \cite{Script_aut}. 

As we did for the case of level $n=p^e$, with $p$ prime in \Cref{thm:powerprimes}, the result on Cartan modular curves of composite level can be sharpened, with our techniques, for levels with a specific type of factorization. However, certain cases remain out of the reach of the strategy described in this work: for example when we are not able to complete the argument using the criterion of \cite{GonzConst} and either we have a low gonality lower bound for the modular curve (e.g., $\Xns(16)$, $\Xns^+(27)$, $\Xns(27)$, $\Xns^+(32)$, $\Xns(32)$, $\Xns^+(64)$) or its jacobian has a large CM part relative to its dimension (see \Cref{rem:48} for the example with the lowest level). A table with the relevant data for the totally split or totally non-split curves of level $n\le 64$, and the description of a few cases having exceptional automorphisms, can be found in the Appendix.

\section{Modular curves}\label{Sec:ModularCurves}
Let $n$ be a positive integer. We denote by $Y(n)$ the (coarse if $n<3$) moduli space that parametrizes pairs $(E, \phi)$ where $E$ is an elliptic curve over a $\Q$-scheme $S$ and $\phi\colon (\Z/n\Z)_S^2 \to E[n]$ is an isomorphism of $S$-group schemes. We denote by $X(n)$ the compactification of $Y(n)$ and we call $X(n)$ the \emph{modular curve of full level $n$}.

Every matrix $\gamma \in \GLn$ gives an automorphism of the constant group scheme $(\Z/n\Z)_S^2$, hence $\gamma$ acts on $Y(n)$ sending $(E, \phi)$ to $(E, \phi \circ \gamma )$. This defines an action of $\GLn$ on $Y(n)$ that extends uniquely to $X(n)$. For each subgroup $H$ of $\GLn$, let $X_H$ be the quotient $X(n)/H$.
By \cite[IV.6.7]{DelRap}, $X_H$ has good reduction over each prime that does not divide $n$ and the smooth model of $Y_H= Y(n)/H$ over $\Z[1/n]$ is a coarse moduli space for \emph{elliptic curves with $H$-structure}, i.e., the equivalence classes of pairs $(E, \phi)$ where $E$ is an elliptic curve over a $\Z[1/n]$-scheme $S$ and $\phi\colon (\Z/n\Z)_S^2 \to E[n]$ is an isomorphism of $S$-group schemes, and the equivalence relation is given by:
\begin{equation}\label{eq_equivalence_relation}
(E,\phi) \sim_H (E', \phi') \iff  (\phi')^{-1} \circ \iota|_{E[n]} \circ \phi = h, \text{ for some } h \in H \text{ and } \iota\colon E \overset{\sim }{\to} E'.
\end{equation}
In particular, for every algebraically closed field $K$ of characteristic $p \nmid n$, we have a bijection between $Y_H(K)$ and the set of elliptic curves over $K$ with $H$-structure. Note that Equation (\ref{eq_equivalence_relation}), for fields with characteristic $p \nmid n$, means that $(E,\phi) \sim_H (E', \phi')$ if and only if the matrix associated to the action of $\iota|_{E[n]} $ relatively to the $\Z/n\Z$-bases of $E[n]$ and $E'[n]$ defined via $\phi$ and $\phi'$, respectively, belongs to $H$.
\begin{rem}\label{rem:pmH}
Since $-1$ is an automorphism of every elliptic curve, then for every $H$, the curve $X_H$ is isomorphic to $X_{\pm H}$, where $\pm H := \{\pm \tx{Id}\} {\cdot} H < \GLn$. Hence, the equivalence relation (\ref{eq_equivalence_relation}) can be written as follows
\begin{equation*}
(E,\phi) \sim_H (E', \phi') \iff  (\phi')^{-1} \circ \iota|_{E[n]} \circ \phi = h, \text{ for some } h \in \pm H \text{ and } \iota\colon E \overset{\sim }{\to} E'.
\end{equation*}
\end{rem}

Let $\HH$ be the complex upper half-plane $\{ \tau \in \C: \mathrm{Im}(\tau)>0 \}$, let $\HH^\pm = \C - \R$ and moreover let $\ol{\HH} = \HH \cup \PP^1(\Q)$ and $\ol{\HH}^\pm = \HH^\pm \cup \PP^1(\Q)$ be their ``compactifications''. The group $\tx{GL}_2(\Z)$ acts on $\HH^\pm$ and $\ol{\HH}^\pm$ by M\"{o}bius transformations. Moreover, every $g$ in $\tx{GL}_2(\Z)$ acts on pairs $(z,\gamma H)\in\HH^\pm \times(\GLn/H)$ as $(g(z), \bar{g}^{-T} \gamma H)$, where $g(z)$ is the image of $z$ under the M\"{o}bius  transformation given by $g$ and $\bar{g}^{-T}$ is the transpose of the inverse of the reduction $\bar g$ of $g \bmod n$. This action gives canonical isomorphisms of Riemann surfaces
\begin{align}
\tx{GL}_2(\Z) \backslash \big(\HH^\pm \times(\GLn/H)\big) \lto Y_H(\C), \label{eq_complex_points_can} \\ 
\tx{GL}_2(\Z) \backslash \big(\ol{\HH}^\pm \times(\GLn/H)\big) \lto X_H(\C).\label{eq_complex_points_can_comp}
\end{align}
The isomorphism (\ref{eq_complex_points_can}) is equivalent to that one described in \cite[IV.5.3]{DelRap} and is given by $\tx{GL}_2(\Z)(\tau,\gamma H)\mapsto (E_{\tau},\phi_{\tau}\circ\gamma ),$ where $E_\tau$ is the elliptic curve $\C/(\Z {+} \Z\tau)$ and $\phi_\tau\colon (\Z/n\Z)^2_\C \to E_\tau[n]$ is the unique isomorphism such that
\[
\phi_\tau\begin{pmatrix}1 \\ 0\end{pmatrix}=\frac{\tau}{n}, \quad\phi_\tau\begin{pmatrix}0 \\ 1\end{pmatrix}=\frac{1}{n}.
\]
We notice that for each $g=\begin{pmatrix}a & b \\ c & d\end{pmatrix}\in\tx{GL}_2(\Z)$ and for each $\gamma\in\tx{GL}_2(\Z/n\Z)$, the two pairs $(E_{\tau},\phi_{\tau}\circ\gamma )$ and $(E_{g(\tau)},\phi_{g(\tau)}\circ\bar{g}^{-T}\gamma )$ are equivalent because the map $z \mapsto (c\tau+d)z$ gives an isomorphism $\iota\colon E_{g(\tau)}\to E_\tau$ such that $\phi_\tau^{-1} \circ\iota \circ \phi_{g(\tau)} = \ol g^T$. Consequently, we obtain an action of $\tx{GL}_2(\Z)$ on pairs $(E_{\tau},\phi_{\tau}\circ\gamma)$ which is the same as the action of \Cref{eq_complex_points_can} and \Cref{eq_complex_points_can_comp}. We notice that the transposition is necessary to make all the maps and the actions compatible.

The isomorphism (\ref{eq_complex_points_can_comp}) is just the extension of the previous one to the compactifications. For each subgroup $H$ of $ \GLn$, we define
\[
\Gamma_{H}:= \{ \gamma \in \tx{SL}_2(\Z): \gamma^T\!\!\!\pmod n \tx{ lies in }H\}.
\]
If $\det H \neq (\Z/n\Z)^\times $, then $X_H(\C)$ is not connected: the number of connected components is $[(\Z/n\Z)^\times:\det(H)]$ and, for each connected component $X_H^{cc}(\C)$, there are isomorphisms of Riemann surfaces
\begin{align}\label{eq_complex_points_non_can}
\Gamma_{gHg^{-1}} \backslash \ol{\HH}\lto X_H^{cc}(\C), \quad \Gamma_{gHg^{-1}} \backslash \HH &\lto Y_H^{cc}(\C), \\
\Gamma_{gHg^{-1}}\tau&\mapsto (E_{\tau},\phi_{\tau}\circ g), \nonumber 
\end{align}
for some $g$ in $\GLn$. In particular, if $\det H = (\Z/n\Z)^\times$, then $Y_H$ and $X_H$ are geometrically connected curves defined over $\Q$. 

The following proposition about the isomorphism (\ref{eq_complex_points_can}) is used in Section \ref{Sec:Aut}.
\begin{prop}\label{PropRamPoint}
Let $n$ be a positive integer, let $H$ be a subgroup of $\GLn$, let $g \in \GLn$ and consider the composition 
\[
\begin{tikzcd}
\HH \arrow[bend right, rr]\arrow[r] &\Gamma_{gHg^{-1}}\backslash \HH  \arrow[hook, r] &Y_H(\C) ,
\end{tikzcd}
\]
where the left map is the natural projection and the right map is in (\ref{eq_complex_points_non_can}). Then
 a point $(E,\phi) \in Y_H(\C)$ is a branch point for such composition if and only if there is a non-trivial automorphism $u$ of $E$ such that $\phi^{-1} {\circ} u|_{E[n]} {\circ} \phi\in \pm H$. If this happens, then each point $\tau \in \HH$ projecting to $(E,\phi)$ has ramification index $\# \Aut(E)/2$.
\end{prop}
\begin{proof}
By Remark \ref{rem:pmH} we can suppose that $H$ contains $-\tx{Id}$. 
Instead of looking at a map $\HH \to Y_H(\C)$ parametrizing a single 
component of $Y_H$, we can work with the canonical map
\begin{equation*}\label{eq:natural_ramified_covering}
\begin{tikzcd}
\HH^\pm \times \GLn \arrow[bend right, rr]\arrow[r,"\pi"] & Y(n)(\C) \arrow[r, "\pi_H"] &Y_H(\C).
\end{tikzcd}
\end{equation*}
Up to substituting $n$ with $3n$ and $H$ with its preimage under $\GL(\Z/3n\Z) \to \GLn$, we can suppose that $n\geq 3$. This implies that $\pi$ is an (unramified) covering map, hence the ramification index of the $\pi_H \circ \pi$  in a point $(\tau,\gamma)$ is equal to the ramification index of $\pi_H$ in the point $\pi(\tau,\gamma)$. Hence, we only need to look at the ramification points of $\pi_H$.
	A point $(E,\phi)\in Y_H(\C)$ is a branch point for $\pi_H$ if and only if the fiber $\pi_H^{-1}(E,\phi)$ has cardinality smaller than $\deg \pi_H = \#  H/2$. The modular interpretation of $Y_H$ and $Y(n)$ implies that
	\begin{equation} \label{eq_fiber_Yn_to_YH}
	\pi_H^{-1}(E,\phi) = \big \{(E,u|_{E[n]} {\circ} \phi {\circ} h) :  h \in H,  u \in \Aut(E) \big\} / \Aut(E),
	\end{equation}
	where $v\in \Aut(E)$ acts sending $(E,\psi)$ to $(E,v|_{E[n]}{\circ} \psi)$. Since $n \ge 3$, the map that sends $u$ to $\phi^{-1} {\circ} u|_{E[n]} {\circ} \phi$ gives an inclusion $\Aut(E) \hookrightarrow \GLn$, hence, by (\ref{eq_fiber_Yn_to_YH}), we have
	\[ 
	\# \pi_H^{-1}(E,\phi) =  \# \Big((H{\cdot}\Aut(E) )/ \Aut(E)\Big)= \# \Big( H/(H{\cap}\Aut(E)) \Big).
	\]
	The group $\Aut(E)$ always contains the multiplication by $-1$ and is cyclic of order $2,4$ or~$6$. Finally, there are two options for $\Aut(E) {\cap} H$: 
	\begin{itemize}
		\item $\Aut(E) {\cap} H$ only contains $\pm \text{Id}$ and $(E,\phi)$ is not a branch point;
		\item $\Aut(E) {\cap} H$ has order equal to $\# \Aut(E) >2$, in this case $(E,\phi)$ is a branch point and, since the map $\pi_H$ is Galois, every point in $\pi_H^{-1}(E,\phi)$ has ramification index $\deg(\pi_H)/\# \pi_H^{-1}(E,\phi) = \# \Aut(E)/2$.
	\end{itemize} 
\end{proof}

\begin{rem}\label{rem:level1}
Notice that all the statements of \Cref{Sec:ModularCurves} and \Cref{Sec:Hecke}  are presented for a subgroup $H$ of $\GLn$, which is not defined for $n=1$. However, we can deduce all the same conclusions for the curve $X(1)$ since it is isomorphic to $X_H$ with, for example, $n=3$ and $H=\GLn$ (in this case, of course $n$ is not the level of the modular curve in the usual sense).
\end{rem}

\section{Hecke operators}\label{Sec:Hecke}

Let $n$ be a positive integer and let $H$ be a subgroup of $\GLn$. For every prime $\ell\nmid n$, there is a divisor $D_\ell\subset X_H \times X_H$ inducing the $\ell$-th Hecke operator
\[
T_\ell \colon \tx{Div}(X_H) \to \tx{Div}(X_H) , \quad T_\ell\colon \tx{Jac}(X_H) \to \tx{Jac}(X_H).
\]
On $Y_H(\C)$, it is described by
\begin{equation} \label{eq:obvious_formula_Tl}
T_\ell(E,\phi) = \sum_{{0}\lneq  C \lneq E[\ell]} (E/C,  \pi_C\circ \phi ),
\end{equation}
where $\pi_C\colon E\to E/C$ is the natural projection. Now we recall the definition of $T_\ell$. Let $H_\ell$ be the subgroup of  $\GL(\Z/n\ell\Z)$ containing the matrices whose reduction modulo $n$ lies in $H$ and whose reduction modulo $\ell$ is an upper triangular matrix. Given a $ \Z[\frac{1}{n\ell}]$-scheme $S$ and an elliptic curve $E/S$ with $H_\ell$-structure $\phi\colon (\Z/n\ell\Z)^2 \to E[n\ell]$, we have two ways of constructing an elliptic curve over $S$ with $H$-structure:
\begin{itemize}
\item The $n$-torsion subgroup of $(\Z/n\ell\Z)^2$ is canonically isomorphic, via the Chinese Remainder Theorem, to $(\Z/n\Z)^2$ and the restriction of $\phi$ to this subgroup gives an isomorphism $\phi|_{(\Z/n\Z)^2}\colon (\Z/n\Z)^2\to E[n]$. One can check that the class of $(E,\phi|_{(\Z/n\Z)^2})$ modulo $\sim_H$ does not depend on the choice of the representative $(E,\phi)$ in the equivalence class defined by $\sim_{H_\ell}$, hence
\[
\tx{pr}(E,\phi ) := (E,\phi|_{(\Z/n\Z)^2} )
\] 
is a well defined elliptic curve over $S$ with $H$-structure.
\item The subgroup  $C\subset E[\ell]$ generated by $ \phi{ \left( \begin{smallmatrix} n\\ 0 \end{smallmatrix}\right) }$ is a subgroup of $E$ of order $\ell$ and $E/C$ is an elliptic curve over $S$. Denoting by $\pi_C  \colon E \to E/C$ the natural projection, we have that
\[
\tx{qt}(E,\phi ) := (E/C,\pi_C \circ \phi|_{(\Z/n\Z)^2})
\]
is a well defined elliptic curve over $S$ with $H$-structure.
\end{itemize}
These two constructions define natural transformations between the functor of elliptic curves with $H_\ell$-structure and the functor of elliptic curves with $H$-structure restricted to schemes over $\Z[\frac1{n\ell}]$. We get induced morphisms between the coarse moduli spaces $Y_{H_\ell}$ and $(Y_H)_{\Z[\frac 1 {n\ell}]}$ that can be extended by smoothness to the compactifications:
\[
\tx{pr}, \tx{qt} \colon X_{H_\ell} \lto (X_H)_{\Z[\frac 1 {n\ell}]}.
\]
The image of $X_{H_\ell}$ under the map $(\tx{pr}, \tx{qt})$ defines a divisor inside $(X_H)_{\Z[\frac 1 {n\ell}]} {\times} (X_H)_{\Z[\frac 1 {n\ell}]}$. Since $X_H$ is smooth over $\Z[\frac 1 {n}]$, this divisor extends uniquely to $D_\ell \subset X_H {\times} X_H$ whose irreducible components project surjectively on each factor $X_H$. This correspondence induces the operator $T_\ell = \tx{qt}_* \circ \tx{pr}^* $ and the definitions of $ \tx{qt}$ and $\tx{pr}$ imply the equality (\ref{eq:obvious_formula_Tl}).

The reduction of $T_\ell$ modulo $\ell$ is described by a celebrated theorem of Eichler and Shimura.
To state this theorem in the full generality, we recall the definition of diamond operators. Let $a \in (\Z/n\Z)^\times$, then the matrix $\smt{a}{0}{0}{a}$ normalizes $H$, hence 
\[
\langle a \rangle (E, \phi) := \left(E, \phi \circ \smt{a}{0}{0}{a} \right)
\]
defines an automorphism of the functor of elliptic curves with $H$-structure. So $\langle a \rangle$ induces an automorphism of the coarse moduli space $Y_H$ and  it extends to an automorphism of the compactification $X_H$. Eichler-Shimura Relation is nowadays a common knowledge, but in the literature is often stated in a different form than we need. The proof of \cite[Theorem 8.7.2]{DS} can be directly adapted to our case, and another proof is in \cite[Theorem 7.9 and Corollary 7.10]{Shi}. We use the result in the following form.
\begin{namedtheorem}[Theorem (Eichler-Shimura Relation)]\label{EichlerShimura}
Let $n$ be a positive integer, let $H$ be a subgroup of $\GLn$, let $\ell$ be a prime number not dividing $n$, let $\ol{X}_H$ be the reduction of $X_H$ modulo $\ell$, let $\ol{T}_\ell,\ol{\langle \ell \rangle}\colon \tx{Div}(\ol{X}_H)\to \tx{Div}(\ol{X}_H)$ be the reduction of the Hecke operator $T_\ell$ and of the diamond operator $\langle \ell \rangle$ and let $\tx{Frob}_\ell \colon \ol{X}_H\to \ol{X}_H $ be the Frobenius morphism. Then
\[
\ol{T}_\ell = (\tx{Frob}_\ell)_*+ \ol{\langle\ell \rangle}_* \circ (\tx{Frob}_\ell)^* .
\]
\end{namedtheorem}
Notice that in general $X_H$ is not geometrically connected and if $X'$ is a connected component of $\ol{X}_H$, the Frobenius morphism $\ol{X}_H \to \ol{X}_H$ may not restrict to a morphism $X'\to X'$. Analogously, if $x$ is a point on $X'$, the divisor $T_\ell(x)$ may be not supported on $X'$. We are interested in Eichler-Shimura Relation because, as already pointed out in \cite[Lemma~$2.6$]{KM}, it implies that, in certain cases, the automorphisms of modular curves automatically commute with all Hecke operators $T_\ell$.

\begin{prop}\label{uandTlprop}
Let $n$ be a positive integer, let $H<\GLn$ be a subgroup containing the scalar matrices and such that $\det H = (\Z/n\Z)^\times$. Let $\ell$ be a prime not dividing $n$ and let $\sigma \in \Gal(\ol\Q/\Q)$ be a Frobenius element at $\ell$. Then, for any automorphism $u$ of $X_H$ defined over a compositum of  quadratic fields, in $\End(\mathrm{Jac}(X_H))$ we have
\begin{equation}\label{Tlu=usigmaTl}
T_\ell \circ u= u^\sigma  \circ T_\ell ,
\end{equation}
where we identify $u$ and $u^\sigma$ with their pushforward on $\mathrm{Jac}(X_H)$.
Moreover, if the gonality of $X_H(\C)$ is greater than $2(\ell+1)$, then (\ref{Tlu=usigmaTl}) holds at level of divisors.
\end{prop}
\begin{proof}
Let $J:=\mathrm{Jac}(X_H)$, let $\tx{Frob}_\ell \colon \ol{X}_H\to \ol{X}_H $ be the Frobenius morphism and let $\phi_\ell$ be the Frobenius generator of $\Gal(\ol{\F}_\ell/\F_\ell)$. Let $D\in\tx{Div}(\ol{X}_H)$ and let $\bar u$ be the reduction of $u$ modulo $\ell$. Using Eichler-Shimura Relation, we have
\begin{align*}
\overline T_\ell \circ \bar u(D)=&\left((\tx{Frob}_\ell)_*+ (\tx{Frob}_\ell)^*\right)\circ \bar u(D)=
(\tx{Frob}_\ell)_*\bar u(D)+(\tx{Frob}_\ell)^*\bar u(D)= \\
=&\bar u^{\phi_\ell}(\tx{Frob}_\ell)_*(D)+\bar u^{\phi_\ell^{-1}}(\tx{Frob}_\ell)^*(D)=\ol{u^{\sigma}}(\tx{Frob}_\ell)_*(D)+\ol{u^{\sigma^{-1}}}(\tx{Frob}_\ell)^*(D).
\end{align*}
Now, since $u$ is defined over a compositum of quadratic fields, the Galois automorphisms $\sigma$ and $\sigma^{-1}$ 
act in the same way on $u$. This implies that the last term in the previous chain of equalities is equal to $\ol{u^{\sigma}}\circ \overline T_\ell(D)$ obtaining $\overline T_\ell\circ \bar u=\ol{u^{\sigma}}\circ \overline T_\ell$ in $\End(J_{\F_\ell})$.

 Since $J$ has good reduction at $\ell$, the natural map $\End(J) \to \End(J_{\F_\ell})$ is injective, hence (\ref{Tlu=usigmaTl}) holds in $\End(J)$. This means that, for any two points $P$ and $Q$ in $X_H(\mathbb C)$, the divisor 
$D:=(T_\ell u-u^\sigma T_\ell)(P-Q)$
is principal. Hence, either $D$ is the zero divisor or is the divisor of a non-constant rational function on $X_H$ of degree at most $2(\ell+1)$.

Now we suppose that the gonality of $X_H$ exceeds $2(\ell+1)$. In this case, there are no non-constant rational functions on $X_H$ of degree at most $2(\ell+1)$, hence $D$ is the zero divisor. This gives the following equality of divisors:
$$T_\ell u(P)+u^\sigma T_\ell (Q)=u^\sigma T_\ell (P) +T_\ell u(Q).$$
For every point $P$, we can choose $Q$ such that the supports of $T_\ell u(P) $ and $T_\ell u (Q)$ are disjoint, and, therefore, last equality implies 
$T_\ell u(P)= u^\sigma T_\ell(P)$ as divisors. Up to a base change to $\C$, each divisor on $X_H$ is a sum of points with integer coefficients, hence we conclude that (\ref{Tlu=usigmaTl}) holds at level of divisors.
\end{proof}
 
\subsection*{Multiple points in the image of Hecke operators}
In the proofs of Section \ref{Sec:Aut} we look at points $P \in X_H(\C)$ and primes $\ell$ such that $T_\ell(P)$ is not a sum of distinct points. In this subsection we study this phenomenon. We denote by $\rho=e^{\frac{2\pi i}{3}}$ the primitive third root of unity contained in $\HH$. Moreover, for every $\tau\in\HH$, we denote by $E_\tau$ the elliptic curve $\C/(\Z {+} \Z\tau)$. The main result is the following 

\begin{thm}\label{theo:repetitions}
	Let $n$ be a positive integer, let $H$ be a subgroup of $\GLn$ and let $\ell \ge 5$ be a prime not dividing $n$. Then, for all points $P \in X_H(\C)$, we have that:
	\begin{enumerate}
		\item in $T_\ell(P)$ there is a point with multiplicity at least $4$ if and only if $P$ is a cusp;
		\item in $T_\ell(P)$ there is a point with multiplicity $3$ if and only if $P=(E_\rho,\phi)$ for some $\phi$ such that and the matrix $\phi^{-1}{\circ}\rho|_{E_{\rho}[n]}{\circ}\phi$ lies in $\pm H$;
		\item in $T_\ell(P)$ there are two distinct points with multiplicity $2$ if and only if $P=(E_i, \phi)$  for some $\phi$ such that the matrix $\phi^{-1}{\circ}i|_{E_i[n]}{\circ}\phi$ lies in $\pm H$.
	\end{enumerate}
\end{thm}
\begin{proof}
	This immediately follows from the following Propositions \ref{TlonCusps},  \ref{prop:rhoi_belli} and \ref{prop:Erhoi_Tl}.	
\end{proof}

 When $P$ is a cusp, we have the following result.
\begin{prop}\label{TlonCusps}
	Let $n$ be a positive integer and let $H$ be a subgroup of $\GLn$. Let $\ell$ be a prime number not dividing $n$, let $\sigma\in\mathrm{Gal}(\overline{\Q}/\Q)$ be a Frobenius element at $\ell$ and let $C\in X_H(\overline{\Q})$ be a cusp. Then, at the level of divisors we have $$T_\ell (C)=C^\sigma+\ell\,\langle \ell \rangle (C^{\sigma^{-1}}).$$
\end{prop}
\begin{proof}
	The divisor $T_\ell(C) = \tx{qt}_* \tx{pr}^*(C) $ is supported on the cusps because the maps $\tx{pr}, \tx{qt}\colon X_{H_\ell} \to X_H $ send non-cuspidal points to non-cuspidal points and cusps to cusps. If we fix a prime ideal $\mathfrak{l} $  in the  algebraic integers such that $\mathfrak{l} \mid \ell$, then, by \cite[IV.3.4]{DelRap}, each cusp in $X_H(\ol \Q)$ reduces to a different point modulo $\mathfrak l$. Thus, it is enough to prove that $T_\ell (C)$ is congruent to $C^\sigma+\ell\, \langle \ell \rangle(C^{\sigma^{-1}})$ modulo $\mathfrak l$, and this is true by Eichler-Shimura Relation.
\end{proof}

We need a criterion to characterize the points $(E,\phi) \in Y_H(\C)$ such that their image via $T_\ell$ contains a point with multiplicity at least $2$. It is given by the following lemma.
\begin{lem}\label{critdblpts} 
Let $n$ be a positive integer, let $H$ be a subgroup of $\GLn$ and let $\ell$ be a prime not dividing $n$. For all points $(E,\phi), (E',\phi') \in Y_H(\C)$ and all positive integers $m\ge 2$, the following are equivalent:
	\begin{enumerate}
		\item\label{critdblptsdef} $T_\ell(E,\phi)$ contains $(E',\phi')$ with multiplicity $m$;
		\item\label{critdblptsalphai} there are $m$ isogenies $\alpha_1, \ldots, \alpha_m \colon E \to E'$ of degree $\ell$ with distinct kernels such that $(\phi')^{-1}{\circ}\alpha_j|_{E[n]}{\circ}\phi$ lies in $\pm H$, for every $j=1,\ldots,m$;
		\item\label{critdblptsbetai} there are $m$ endomorphisms $\beta_1=\ell, \beta_2,\ldots, \beta_m$ of $E'$ of degree $\ell^2$ and an isogeny $\alpha \colon E' \to E$ of degree $\ell$ such that:
	\end{enumerate}
	\begin{description}
		\item[P1]\label{critdblptsbetau} $\beta_i \neq u \circ \beta_j$, for $i,j=1,\ldots,m$, such that $i\neq j$ and for each $u \in \Aut(E')$;
		\item[P2]\label{critdblptsker} $\ker \alpha \subset \ker \beta_j$, for every $j=1,\ldots,m$;
		\item[P3]\label{critdblptsH} the matrices $\ell^{-1}\phi^{-1} {\circ} \alpha|_{E'[n]} {\circ} \phi'$ and $\ell^{-1} (\phi')^{-1} {\circ} \beta_j|_{E'[n]} {\circ} \phi'$ lie in $\pm H$, for every $j=1,\ldots,m$, where $\ell^{-1}$ is the inverse of the scalar matrix $\ell \bmod n$.
	\end{description}
\end{lem}
\begin{proof}
The equivalence between (\ref{critdblptsdef}) and (\ref{critdblptsalphai}) follows by definition of Hecke operator. Now we prove the equivalence between (\ref{critdblptsalphai}) and (\ref{critdblptsbetai}). Let $\alpha_1, \ldots, \alpha_m$ be isogenies of degree $\ell$ with distinct  kernels, then it is enough to take $\alpha$ equal to the dual of $\alpha_1$ and $\beta_j= \alpha_j {\circ} \alpha$, for $j=1,\ldots,m$. Conversely, if $\beta_1, \ldots, \beta_m$ respect the three properties above, then, for every $j=1,\ldots,m$, we can take $\alpha_j$ to be the unique isogeny such that $\beta_j = \alpha_j{\circ}\alpha$.
\end{proof}

In the next two propositions we study some cases in which $T_\ell(E,\phi)$ contains points with higher multiplicity, with a particular attention to $E_i$ and $E_\rho$. Namely, \Cref{prop:Erhoi_Tl} characterizes when $\phi^{-1}{\circ}\tau|_{E_{\tau}[n]}{\circ}\phi$ belongs to $\pm H$, for $\tau =\rho,i$, in terms of the multiplicities shown in the divisor $T_\ell(E_\tau,\phi)$, while \Cref{prop:rhoi_belli} proves that if $T_\ell(E,\phi)$ shows certain multiplicities, then $E$ has complex multiplication by $\Q(i)$ or $\Q(\rho)$.

\begin{prop}\label{prop:rhoi_belli}
	Let $n$ be a positive integer, let $H$ be a subgroup of $\GLn$, let $\ell$ be a prime not dividing $n$ and let $(E, \phi)$ be a $\C$-point of $Y_H$. Then:
\begin{enumerate}
\item\label{prop:mult3} the points in the image $T_\ell(E,\phi)$ have multiplicity at most $3$;
\item\label{prop:rho_bello} if $T_\ell(E,\phi)$ contains a point with multiplicity $3$, then $E \cong E_{\rho}$;
\item\label{prop:i_bello} if there are two distinct points of $Y_H(\C)$ appearing with multiplicity $2$ in $T_\ell(E,\phi)$, then $E \cong E_i$.
\end{enumerate}
\end{prop}
\begin{proof}
	Parts (\ref{prop:mult3}) and (\ref{prop:rho_bello}).

First we prove that if $T_\ell(E,\phi)$ contains a point with multiplicity \emph{at least} $3$, then $E \cong E_{\rho}$. Let $(E',\phi')\in Y_H(\C)$ be such that $T_\ell(E,\phi)$ contains $(E',\phi')$ with multiplicity at least~$3$. \Cref{critdblpts} implies the existence of isogenies $\alpha_1, \alpha_2, \alpha_3\colon E \to E'$ of degree $\ell$ with different kernels.
For each $j\neq k$, we define $\gamma_{j,k} := \hat \alpha_j \circ \alpha_k \in \End(E)$, where $\hat \alpha_j$ is the dual of $\alpha_j$. Since $\alpha_k \neq \pm \alpha_j$, for all $j\neq k$, then 
\begin{equation}\label{eq:composition_alphas}
\gamma_{j,k} = \hat\alpha_j \circ \alpha_k \neq \pm \hat\alpha_j\circ \alpha_j = \pm \ell \in \End(E),
\end{equation}
so $\gamma_{j,k}$ is cyclic of degree $\ell^2$. In particular $E$ has more than $2$ endomorphisms of degree $\ell^2$, hence it has complex multiplication over some imaginary quadratic field $K$. We suppose by contradiction that $\Aut(E)=\{\pm 1\}$. In this case \Cref{eq:composition_alphas} implies that $\ker (\gamma_{j,k}) \neq \ker(\ell) = E[\ell]$ and consequently that  $\ker (\gamma_{j,k}) \cap E[\ell] $ has order $\ell$. Hence, for all $j \neq k$, we have that $\ker (\gamma_{j,k})\cap E[\ell] = \ker(\alpha_k)$. In particular, the $12$ endomorphisms $\pm \gamma_{j,k}$, with $j \neq k$, are pairwise distinct: if we had $ \hat\alpha_j\circ \alpha_k = \pm  \hat\alpha_r\circ \alpha_m$, then 
\[
\ker (\alpha_k )= \ker (\hat\alpha_j {\circ }\alpha_k)\cap E[\ell] =  \ker (\hat\alpha_r {\circ }\alpha_m)\cap E[\ell] = \ker(\alpha_m) ,
\]
implying $k=m$ and, consequently, $j=r$. Let $\calO_K$ be the ring of integers of $K$ and let $m$ be the unique positive integer such that $\End(E)=\Z + m \calO_K$. If $\ell \mid m$, the only elements of $\Z + m \calO_K$ having norm divisible by $\ell$ belong to $\ell \calO_K$, hence all the elements of $\End(E)$ having degree $\ell^2$ have form $\ell u$, for $u\in \calO_K^\times$. Hence there are at most $6$ of such elements. If $\ell\not \mid m$, the ideals of norm $\ell^2$ inside $\End(E)$ are in bijection with the ideals of norm $\ell^2$ inside $\calO_K$. Hence, by looking at the possible factorizations, there are at most $3$ of such ideals and therefore at most $6$ elements of $\End(E)$ of degree $\ell^2$. In all the cases there are at most $6$ elements of $\End(E)$ of degree $\ell^2$, implying that the elements $\pm \gamma_{j,k}$ are not distinct, which is a contradiction. 

We excluded the cases where $\Aut(E) = \{\pm1 \}$, it remains to exclude the case $E\cong E_i$. We suppose $E\cong E_i$ and we look at $\delta_k := \gamma_{k,3}$, for $k=1,2$. Since $\pm\delta_2, \pm \delta_1, \pm \ell$ are distinct, then at least one of the $\delta_k$'s is not contained in $\{ \pm \ell, \pm i\ell\}$ and, up to renaming, we can suppose that this happens for $\delta_2$. Hence, we can factor  $\delta_2 = i^r \lambda^2$, for some integer $r$ and some prime element $\lambda \in \Z[i]=\End(E)$. We deduce that $\ker(\alpha_3) = \ker(\delta_2) \cap E[\ell] = \ker(\lambda)$ and consequently
\[
E' \cong E/\ker(\alpha_3) =E/\ker(\lambda) \cong E_i.
\]
Up to units of $\End(E)$, there are at most $2$ elements of $\End(E) \cong \text{Hom} (E, E')$ of degree~$\ell$, contradicting the existence of $\alpha_1, \alpha_2, \alpha_3$ and proving that $E \cong E_{\rho}$.

Finally, we prove that $T_\ell(E,\phi)$ does not contain points with multiplicity greater than~$3$. We suppose by contradiction that $T_\ell(E,\phi)$ contains $(E',\phi')$ with multiplicity at least $4$. Because of the previous step, we have $E\cong E_\rho$. Since, up to unit, there are at most $2$ elements of $\End(E_\rho)$ of degree $\ell$, then there are at most $2$ points of $T_\ell(E_\rho,\phi)$ of the form $(E_\rho, \varphi)$ and consequently $E' \not \cong E_\rho$ which is equivalent to $\End(E') \neq \Z[\rho]$. The isogenies between $E$ and $E'$ give an inclusion $\ell\End(E)\subset \End(E')$, implying that $\End(E') = \Z[\ell\rho]$. Hence the only elements in $\End(E')$ of degree $\ell^2$ are $\pm \ell, \pm \rho \ell, \pm\rho^2\ell$, contradicting the existence of $\beta_1, \beta_2,\beta_3, \beta_4$ as in \Cref{critdblpts}, Condition~(\ref{critdblptsbetai}).  This contradiction concludes the proof of Parts (\ref{prop:mult3}) and (\ref{prop:rho_bello}).

Part (\ref{prop:i_bello}).

Let $(E', \phi'), (E'', \phi'') \in Y_H(\C)$ be such that $T_\ell(E,\phi)= 2(E',\phi') + 2 (E'', \phi'')+D$ where the support of $D$ does not contain neither $(E',\phi')$ nor $(E'', \phi'')$.  By \Cref{critdblpts} there are isogenies $\alpha_1, \alpha_2\colon E \to E'$ and $\alpha_3, \alpha_4\colon E \to E''$ 
 such that the subgroups $\ker(\alpha_i)$, for $i=1,\ldots, 4,$ are four different subgroups of $E[\ell]$ of order $\ell$. By looking at the endomorphisms $\gamma_{j,k} := \hat \alpha_j \circ \alpha_k \in \End(E)$ for $(j,k) \in \{(1,2),(2,1), (3,4), (4,3)\}$, we can exclude the case $\Aut(E)=\{\pm1\}$ with the same arguments used for Parts (\ref{prop:mult3}) and (\ref{prop:rho_bello}). To prove Part (\ref{prop:i_bello}) it remains to exclude the case $E\cong E_{\rho}$. We suppose by contradiction $E\cong E_{\rho}$. Since, up to unit, there are at most $2$ elements of $\End(E_\rho)$ of degree $\ell$, then there are at most $2$ points of $T_\ell(E_\rho,\phi)$ of the form $(E_\rho, \varphi)$. Hence we can suppose $E' \not \cong E_\rho$ which, as in the proof of Parts (\ref{prop:mult3}) and (\ref{prop:rho_bello}), implies $\End(E') = \Z[\ell\rho]$. 
By Lemma \ref{critdblpts}, there are $\alpha$, $\beta_1=\ell$ and $\beta_2$ satisfying {\bf P1}, {\bf P2}, {\bf P3} of \Cref{critdblpts}. In $\Z[\ell\rho]$ the only elements of norm $\ell^2$ have form $\rho^k \ell$, hence $\beta_2$ has the same form   for some $k \in \{1,2,4,5\}$. We define $\beta_3 :=  \rho^{2k} \ell =  \ell^{-1}\beta_2^2$ that satisfies property {\bf P3} of Lemma \ref{critdblpts} and, since  $\End(E') = \Z[\ell\rho]$, the elements $\beta_1, \beta_2, \beta_3$ satisfy the property {\bf P1}. We now prove that $\beta_3$ satisfies property {\bf P2} as well. Since $\ker(\alpha)\subset\ker(\beta_2)$, we can write $\beta_2= \gamma\circ\alpha$ for some isogeny $\gamma\colon E_\rho\to E'$ of degree $\ell$. Notice that, if $\alpha\circ \gamma \in \End (E_\rho)$ was not a multiple of $\ell$, then we would have $\alpha\circ \gamma = u \lambda^2$, for some $u \in \Aut(E_\rho)$ and some element $\lambda \in \End(E_\rho)$ of degree~$\ell$, implying
\[
\ker(\lambda) = \ker(u \lambda^2) \cap E_\rho[\ell] \supset \ker(\gamma).
\]
Since $\ker(\lambda)$ and $\ker(\gamma)$ have the same cardinality, this implies that $\ker(\lambda)=\ker(\gamma)$, and therefore  that $E_\rho \cong E_\rho/\ker(\lambda) = E_\rho/\ker(\gamma) \cong E'$, which is absurd. We deduce that $\alpha\circ \gamma = \ell \delta$ for some  $\delta \in \End (E_\rho)$, hence we can write
\[
\beta_3 = \ell^{-1} \beta_2^2 = \gamma \circ \delta \circ \alpha \quad \implies \quad \ker(\alpha)\subset \ker(\beta_3),
\]
which is  property {\bf P2} of Lemma \ref{critdblpts}. Applying \Cref{critdblpts} once again, we deduce that $(E',\phi')$ appears in $T_\ell(E,\phi)$ with multiplicity~$3$, which is a contradiction.
\end{proof}
\begin{prop}\label{prop:Erhoi_Tl}
	Let $n$ be a positive integer, let $H$ be a subgroup of $\GLn$ and let $\ell$ be a prime not dividing $n$.
\begin{enumerate}
\item\label{prop:Erho_Tl} Let $(E_\rho, \phi)\in Y_H(\C)$. The matrix $\phi^{-1}{\circ}\rho|_{E_{\rho}[n]}{\circ}\phi$ lies in $\pm H$ if and only if the divisor $T_\ell(E_\rho,\phi)$ contains a point with multiplicity $3$.
\item\label{prop:Ei_Tl} Let $(E_i, \phi)\in Y_H(\C)$. If $\ell>2$: The matrix $\phi^{-1}{\circ}i|_{E_i[n]}{\circ}\phi$ lies in $\pm H$  if and only if there are two distinct points of $Y_H(\C)$ appearing with multiplicity $2$ in $T_\ell(E_i,\phi)$.
		If $\ell=2$: The matrix $\phi^{-1}{\circ}i|_{E_i[n]}{\circ}\phi$ lies in $\pm H$ if and only if there are two distinct points $P_1, P_2\in Y_H(\C)$ such that $T_2(E_i, \phi) = 2P_1 + P_2$.
\end{enumerate}
\end{prop}
\begin{proof}
Part (\ref{prop:Erho_Tl}).

If $C\subset E_{\rho}[\ell]$ is a subgroup of order $\ell$, then $\rho C$ and $\rho^2 C$ are subgroups of order $\ell$ as well and there are two unique isomorphisms $u, v$ that make the following diagrams commutative:
	\[
	\begin{tikzcd}
	E_\rho\arrow[d, "\pi_C"] \arrow[r, "\rho"] & E_\rho \arrow[d, "\pi_{\rho C}"]\,, && E_\rho\arrow[d, "\pi_C"] \arrow[r, "\rho^2"] & E_\rho \arrow[d, "\pi_{\rho^2 C}"]  \\
	E_\rho/C \arrow[r, "u"] & E_\rho/\rho  C,&&  E_\rho/C \arrow[r, "v"] & E_\rho/\rho^2  C.
	\end{tikzcd}
	\]
We have that $\rho C = C$ if and only if $\rho$ is an endomorphism of $E_\rho/C$, which is in turn equivalent to $\Aut(E_\rho/C) \neq \{\pm 1\}$  or $\End(E_\rho/C) = \Z[\rho]$ and, since the class number of $\Z[\rho]$ is equal to $1$, this is equivalent to $E_\rho/C\cong E_\rho$. Hence, if $\rho C \neq C$, then $\Aut(E_\rho/C) = \{\pm 1\}$ and, using that $\pi_C$ and $\pi_{\rho C}$ are bijections on the $n$-torsion subgroups, we have
\begin{align}\label{eq:equiv_Erho_ramif}
(E_\rho/C, \pi_C{\circ}\phi) = (E_\rho/\rho C, \pi_{\rho C}{\circ}\phi) \iff &  (\pi_{\rho C}|_{E_{\rho}[n]}{\circ}\phi)^{-1}{\circ} u|_{(E_{\rho}/C)[n]} {\circ} (\pi_{C}|_{E_{\rho}[n]}{\circ}\phi) \in \pm H \nonumber \\
\iff &  \phi^{-1}{\circ} \rho|_{E_{\rho}[n]}{\circ}\phi \in \pm H.
\end{align}
Analogously, $\rho^2 C \neq C$ if and only if $\Aut(E_\rho/C) = \{\pm 1\}$ and when this happens
\begin{equation}\label{eq:equiv_Erho_ramif2}
(E_\rho/C, \pi_C{\circ}\phi) = (E_\rho/\rho^2 C, \pi_{\rho^2 C}{\circ}\phi) \quad \iff \quad  \phi^{-1}{\circ} \rho|_{E_{\rho}[n]} {\circ}\phi \in \pm H.
\end{equation}
The endomorphism $\rho$ does not act as a scalar on $E_\rho[\ell]$: if it acted as a scalar $k$, then $\rho-k=\ell a +\ell b \rho$, with $a,b\in \Z$, implying $\ell b=1$ that is impossible. Hence, there are at most two non-trivial subgroups of $E_\rho[\ell]$ that are $\rho$-stable. In particular we can take a non-trivial subgroup $C_0$ such that $C_0$, $\rho C_0$ and $\rho^2 C_0$ are pairwise distinct.

If $\phi^{-1}{\circ} \rho|_{E_{\rho}[n]} {\circ}\phi$ lies in $\pm H$, then, by (\ref{eq:equiv_Erho_ramif}) and (\ref{eq:equiv_Erho_ramif2}),
\[
T_\ell(E_\rho,\phi) \geq (E_\rho/C_0, \pi_{C_0}{\circ}\phi) {+} (E_\rho/\rho C_0, \pi_{\rho C_0}{\circ}\phi) {+} (E_\rho/\rho^2 C_0, \pi_{\rho^2 C_0}{\circ}\phi)  = 3 (E_\rho/C_0, \pi_{C_0}{\circ}\phi).
\]

Conversely, if  $T_\ell(E_\rho,\phi)$ contains a point with multiplicity $3$, there are three pairwise distinct subgroups $C_1,C_2,C_3 \subset E_{\rho}[\ell]$ of order $\ell$ such that
\[
(E_\rho/C_1, \pi_{C_1}{\circ} \phi) =  (E_\rho/C_2, \pi_{C_2}{\circ} \phi) =  (E_\rho/C_3, \pi_{C_3}{\circ} \phi).
\]
If one of the $C_j$ is $\rho$-stable, then $E_\rho/C_1 \cong E_\rho/C_2  \cong E_\rho/C_3 \cong E_\rho$, and $C_1,C_2,C_3$ are all $\rho$-stable, contradicting that there are at most two non-trivial $\rho$-stable subgroups of $E_\rho[\ell]$. 
In particular $\Z[\rho]\supsetneq \End(E_\rho/C_1)$ and since $E/C_1$ is $\ell$-isogenous to $E_\rho$ we deduce that $\End(E_\rho/C_1) = \Z[\ell\rho]$. Hence, the only endomorphisms of $E_\rho/C_1$ having degree $\ell^2$ are $\pm \ell, \pm \rho \ell, \pm \rho^2 \ell$ and so there are at most three subgroups $C \subset E_\rho[\ell]$ of order $\ell$ such that $E_\rho/C$ is isomorphic to $E_\rho/C_1$, namely: $C_1, \rho C_1$ and $\rho^2 C_1$. We deduce that, up to reordering, $C_2 = \rho C_1$ hence, by (\ref{eq:equiv_Erho_ramif}),
$\phi^{-1}{\circ}\rho|_{E_{\rho}[n]}{\circ}\phi$ lies in $\pm H$.

Part (\ref{prop:Ei_Tl}).

If $C\subset E_i[\ell]$ is a subgroup of order $\ell$, then $i C$ is a subgroup of order $\ell$ as well and there is a unique isomorphism $u$ that makes the following diagram commutative:
	\[
	\begin{tikzcd}
	E_i\arrow[d, "\pi_C"] \arrow[r, "i"] & E_i \arrow[d, "\pi_{iC}"] \\
	E_i/C \arrow[r, "u"] & E_i/iC.
	\end{tikzcd}
	\]
	We have that $i C = C$  if and only if $\End(E_i/C) = \Z[i]$ if and only if $\Aut(E_i/C) \neq \{\pm 1\}$. Hence, if $iC \neq C$, then $\Aut(E_i/C) = \{\pm 1\}$ and, using that $\pi_C$ and $\pi_{iC}$ are bijections on the $n$-torsion subgroups, we have
	\begin{equation}\label{eq:equiv_Ei_ramif}
	\begin{aligned}
	(E_i/C, \pi_C{\circ}\phi) = (E_i/i C, \pi_{iC}{\circ}\phi) \quad \iff & \quad (\pi_{iC}{\circ}\phi)^{-1}{\circ} u|_{(E_i/C)[n]} {\circ} (\pi_{C}{\circ}\phi) \in \pm H \\
	\iff & \quad  \phi^{-1}{\circ} i|_{E_i[n]} {\circ}\phi \in \pm H.
	\end{aligned}
	\end{equation}
	Similarly to the action of $\rho$ on $E_\rho[\ell]$, the endomorphism $i$ does not act as multiplication by a scalar on $E_i[\ell]$, hence there are at most two non-trivial subgroups of $E_i[\ell]$ that are $i$-invariant. In other words, for each subgroup $C\subset E_i[\ell]$ of order $\ell$, except at most two, we have $C \neq iC$.
 If $\ell\ge 5$, since $T_\ell(E_i,\phi)$ contains $\ell{+}1\ge 6$ points counted with multiplicity, we deduce the existence of $C_1, C_2$ such that $C_1,C_2,iC_1,iC_2$ are different cyclic subgroups of $E_i[\ell]$. When $\ell=3$, the same conclusion is true because there is no subgroup of $E_i[3]$ which is invariant under the endomorphism~$i$.
If  $\phi^{-1}{\circ}i|_{E_i[n]}{\circ}\phi$ lies in $\pm H$, then, by the equivalences (\ref{eq:equiv_Ei_ramif}), we have
\[
\begin{aligned}
T_\ell(E_i, \phi) &\ge  (E_i/C_1, \pi_{C_1} {\circ} \phi) + (E_i/iC_1, \pi_{iC_1} {\circ} \phi) + (E_i/C_2, \pi_{C_2} {\circ} \phi) + (E_i/iC_2, \pi_{iC_2} {\circ} \phi)= \\
&=  2(E_i/C_1, \pi_{C_1} {\circ} \phi) +2 (E_i/C_2, \pi_{C_2} {\circ} \phi).
\end{aligned}
\]
Moreover $(E_i/C_1, \pi_{C_1} {\circ} \phi)$ and $(E_i/C_2, \pi_{C_2} {\circ} \phi)$ do not appear with multiplicity greater than~$2$ because of Proposition \ref{prop:rhoi_belli}.

Now we assume that there are two distinct points $(E_i/C_1, \pi_{C_1} {\circ} \phi)$ and $(E_i/C_2, \pi_{C_2} {\circ} \phi)$ in $Y_H(\C)$ appearing with multiplicity $2$ in $T_\ell(E_i,\phi)$. Since there are at most two subgroups $C \subset E_i[\ell]$ such that  $C = iC$, then there are at most two points, counted with multiplicity, in $T_\ell(E_i, \phi)$ having form $(E_i, \psi)$. Hence, up to reordering, we have $E_i/C_1 \not \cong E_i$, or equivalently $iC_1 \neq C_1$. Thus $\End(E_i/C_1) = \Z[\ell i]$, and this implies that $\pm \ell$ and $\pm \ell i$ are the only elements of $\End(E_i/C_1)$ having degree $\ell^2$. Thus, applying  \Cref{critdblpts} by checking Condition \ref{critdblptsbetai} on the modular curve $X(1)$ (see also \Cref{rem:level1}), we see that there is at most one cyclic subgroup $C \subset E_i[\ell]$ such that $E_i/C \cong E_i/C_1$ and $C\neq C_1$. Since $C=iC_1$ has this property and since $(E_i/C_1, \pi_{C_1} {\circ} \phi)$ appears in $T_\ell(E_i, \phi)$ with multiplicity $2$, we have
\[
(E_i/C_1, \pi_{C_1}{\circ}\phi) = (E_i/iC_1, \pi_{iC_1}{\circ}\phi),
\]
and by the equivalences (\ref{eq:equiv_Ei_ramif}), we have that $\phi^{-1}{\circ}i|_{E_i[n]}{\circ}\phi$ lies in $\pm H$.

The case $\ell=2$ can be proven with similar arguments.
\end{proof}

\section{Cartan modular curves and their jacobians}\label{Sec:Cartan}
We give the definition of Cartan modular curves following \cite[Appendix A.5]{SerreMordell}. Let $n{>}1$ be an integer and let $A$ be a free commutative étale $\Z/n\Z$-algebra of rank $2$. For each prime $p\mid n$, we have that $A/pA$ is isomorphic either to $\F_p\times\F_p$ or to $\F_{p^2}$: in the former we say that $A$ is \emph{split} at $p$, in the latter we say that $A$ is \emph{non-split} at $p$. Moreover, for every assignment of each prime $p{\mid}n$ to split or non-split, there is a unique, up to isomorphism, algebra $A$ which is split or non-split at every $p\mid n$ accordingly to the assignment.

We fix a $\Z/n\Z$-basis of $A$ and, consequently, we identify the automorphism group of $A$, as $\Z/n\Z$-module, with $\GLn$. The group $A^\times$ of the units of $A$ acts on $A$ by multiplication, giving an embedding of $A^\times$ inside $\GLn$. A subgroup of $\GLn$ which is the image of such an embedding is called a \emph{Cartan subgroup}. 
The normalizer of $A^\times$ inside $\GLn$ contains all the matrices representing automorphisms of the ring $A$, hence $H := \langle A^\times, \Aut_{\tx{Ring}}(A) \rangle$ is a subgroup of $\GLn$ that contains $A^\times$ as normal subgroup. We call every such an $H$ a \emph{Cartan-plus subgroup} of $\GLn$. The natural map $\Aut_{\tx{Ring}}(A) \to \prod_{p\mid n}\Aut_{\tx{Ring}}(A\otimes \F_p)$ is an isomorphism, hence $\Aut_{\tx{Ring}}(A)$ is isomorphic to $(\Z/2\Z)^{\omega(n)}$, where $\omega(n)$ is the number of prime divisors of $n$. In particular, given $A$, the Cartan subgroup has index $2^{\omega(n)}$ inside the Cartan-plus subgroup. Moreover, if $n$ is odd, the Cartan-plus is equal to the normalizer of the Cartan subgroup inside $\GLn$. We call \emph{Cartan modular curves} the modular curves associated to Cartan subgroups or to Cartan-plus subgroups of $\GLn$.

When $n=p^e$ is a prime power, we use the following notation: 
\begin{itemize}
	\item $X_\ns^+(p^e):=X_H$, if $H$ is a Cartan-plus subgroup non-split at $p$;
	\item $X_\ns(p^e):=X_H$, if $H$ is a Cartan subgroup non-split at $p$;
	\item $\Xs^+(p^e):=X_H$, if $H$ is a Cartan-plus subgroup split at $p$;
	\item $\Xs(p^e):=X_H$, if $H$ is a Cartan subgroup split at $p$.
\end{itemize}
\begin{rem}
	If $H_1$ and $H_2$ are two conjugate subgroups of $\GLn$, then the corresponding modular curves $X_{H_1}$ and $X_{H_2}$ are isomorphic. Moreover, given two Cartan or two Cartan-plus subgroups $C_1$ and $C_2$ of $\GLn$ with the same assignment of each prime $p\mid n$ to split or non-split, then $C_1$ and $C_2$ are conjugate, so $X_{C_1}\cong X_{C_2}$. This implies that the above definitions are unambiguous. 
\end{rem}
\begin{rem}\label{rem:isomconj}
Let $H_1$ and $H_2$ be subgroups of $\GLn$ such that $\Gamma_{H_1}=g\Gamma_{H_2}g^{-1}$ for a suitable $g\in\mathrm{GL}_2(\Q)$ with $\det(g)>0$. In this case there is an isomorphism of Riemann surfaces given by
\begin{align*}
X_{H_1}(\C)&\to X_{H_2}(\C), \\
\Gamma_{H_1}\tau&\mapsto \Gamma_{H_2}g(\tau).
\end{align*}
See \cite[Section 5.1]{DS} for more details about this.
\end{rem}
We want to understand the structure, up to isogeny, of the jacobian of the Cartan modular curves. This is achieved using Chen's isogenies (see \cite{Chen}, \cite{Edix},\cite{ChenPowers}). Let $p$ be a prime and let $e$ be a positive integer. We give an analogous of \cite[Theorem~1.1]{ChenPowers} involving the jacobian of $X_\ns(p^e)$ for every $p$, and, to do this, we extend the analysis in \cite{ChenPowers} to the case $p=2$. In order to state our result, we choose a non-square element $\xi\in (\Z/p^e\Z)^\times$ when $p$ is odd and define the following subgroups of $\textnormal{GL}_2(\Z/p^e\Z)$ for every prime $p$:
\allowdisplaybreaks
\begin{align*}
&C_{\textnormal{s}}(p^e):=\left\{\begin{pmatrix}a & 0 \\ 0 & d \end{pmatrix}, a,d\in(\Z/p^e\Z)^\times \right\}; \\
&C_{\textnormal{s}}^+(p^e):=C_{\textnormal{s}}(p^e)\cup\left\{\begin{pmatrix}0 & b \\ c & 0 \end{pmatrix}, b,c\in(\Z/p^e\Z)^\times \right\}; \\
&C_\ns(2^e):=	\left\{\begin{pmatrix}a & b \\ b & a+b \end{pmatrix}, a,b\in\Z/2^e\Z,  (a,b) \not\equiv (0,0) \bmod 2 \right\}; \\
&C_\ns^+(2^e):=C_\ns(2^e)\cup\left\{\begin{pmatrix}a & a-b \\ b & -a \end{pmatrix}, a,b\in\Z/2^e\Z, (a,b) \not\equiv (0,0) \bmod 2 \right\}; \\
&C_\ns(p^e):=	\left\{\begin{pmatrix}a & b\xi \\ b & a \end{pmatrix}, a,b\in\Z/p^e\Z,  (a,b) \not\equiv (0,0) \bmod p \right\}, \quad \tx{if $p$ is odd}; \\ 
&C_\ns^+(p^e):=C_\ns(p^e)\cup\left\{\begin{pmatrix}a & b\xi \\ -b & -a \end{pmatrix}, a,b\in\Z/p^e\Z, (a,b) \not\equiv (0,0) \bmod p \right\}, \quad \tx{if $p$ is odd}; \\
&B_r(p^e):=\left\{\begin{pmatrix}a & bp^r \\ cp^{r+1} & d \end{pmatrix}, a,b,c,d\in\Z/p^e\Z, \quad ad \not\equiv 0 \bmod p \right\}, \quad \text{for }r=0,1,\ldots,e-1; \\
&T_r(p^e):=\left\{\begin{pmatrix}a & bp^r\\ cp^r & d \end{pmatrix}, a,b,c,d\in\Z/p^e\Z, ad-bcp^{2r} \in (\Z/p^e\Z)^\times \right\}, \quad \text{for }r=0,1,\ldots,e.
\end{align*}
\allowdisplaybreaks[0]
We remark that $T_e(p^e)=C_{\textnormal{s}}(p^e)$ and that $C_{\tx{s}}(p^e), C_\ns(p^e)$ are respectively a split and a non-split Cartan subgroup of $\GL(\Z/p^e\Z)$ and $C_{\tx{s}}^+(p^e), C_\ns^+(p^e)$ are the corresponding Cartan-plus subgroups. 
\begin{prop}\label{lem:reprisom}
	Let $p$ be a prime, let $e$ be a positive integer and let $G=\GL(\Z/p^e\Z)$. We have the following isomorphism of $\Q$-representations of $G$:
	\begin{equation}\label{eq:chenisogeny}
	\Q[G/C_\ns(p^e)] \oplus \bigoplus_{r=0}^{e-1} 2\Q[G/B_r(p^e)] \cong \Q[G/C_{\textnormal{s}}(p^e)] \oplus \bigoplus_{r=0}^{e-1} 2\Q[G/T_r(p^e)].
	\end{equation}
\end{prop}
\begin{proof}
We follow the same strategy of \cite{ChenPowers}. It is enough to prove that the representation on the right hand side has the same character of the representation on the left hand side. For every subgroup $H\subset G$, the character $\chi_H$ of the representation $\Q[G/H]$ is not hard to compute: for each $\gamma \in G$ we have $\gamma(gH) = (\gamma g)H$, hence, with respect to the basis $\{gH \}$, the matrix $M_\gamma$ associated to the action of $\gamma$ is a permutation matrix and consequently
\begin{equation}\label{eq:characters}
\chi_{H}(\gamma) = \text{tr}(M_\gamma)= \# \{ gH: \gamma g H = g H \} = 
\frac{\# \{ g: \gamma g  \in  g H \} }{\# H} = \frac{\# \{ g: g^{-1}\gamma g  \in  H \}}{\# H}.
\end{equation}
Hence to compute $\chi_{H}(\gamma) $ it is enough to compute $\# \{ g: g^{-1}\gamma g  \in  H \}$ and $\# H$. It is enough to compute $\chi_{H}(\gamma) $ for a set of representatives $\gamma$ up to conjugation. We choose the same representatives as in \cite[Table 2]{ChenPowers}.

If $p=2$, the character $\chi_H$ for the groups appearing in the statement is computed in the \nameref{sec:appendix} of this article. If $p$ is odd and $H$ has the form $B_r, T_r$ or $C_{\tx s}$, the character $\chi_H$ is given in \cite[Tables 3 and 4]{ChenPowers}; if $p$ is odd and $H=C_\ns(p^e)$,  then
\[
\chi_H(g)=\begin{cases}
(p{-}1)p^{2e-1}, & \text{if $g$ is a scalar matrix (type $I$ in \cite[Tables 3, 4]{ChenPowers})}, \\
2p^{2\mu}, & \text{if $g$ is a conjugate of $\smt{\alpha}{\xi\beta p^\mu}{\beta p^\mu}{\alpha}$, with $\beta \in (\Z/p^e\Z)^\times$} \\
 & \text{ and $0\le \mu<e-1$ (types $RI'_\mu$ and $T'$ in \cite[Tables 3, 4]{ChenPowers})}, \\
0, & \text{otherwise}.
\end{cases}
\]
The characters of the representations in Equation (\ref{eq:chenisogeny}) are sums of the previous characters. A straightforward computation proves the proposition.
\end{proof}
From the previous result about representations follows a result about jacobians of modular curves.
\begin{prop}\label{prop:Chen_Xns(p^e)}
Let $p$ be a prime, let $e$ be a positive integer and let $J_\ns(p^e)$ be the jacobian of $X_\ns(p^e)$. We have the following isogenies over $\Q$:
\[
J_\ns(p^e)\times \prod_{r=0}^{e-1} J_0(p^{2r+1})^2\sim J_0(p^{2e})\times \prod_{r=0}^{e-1} J_0(p^{2r})^2, \qquad J_\ns(p^e) \sim \prod_{r=1}^{e} J_0^{\tx{new}}(p^{2r}).
\]
\end{prop}
\begin{proof}
For every $r=0,\ldots,e-1$, let 
$w_{p^r}:=\smt{0}{-1}{p^r}{0}$, then we have
\begin{align*}
w_{p^{r}}\Gamma_{B_r(p^e)}w_{p^{r}}^{-1}=\Gamma_0(p^{2r+1}) \qquad \text{and} \qquad w_{p^r}\Gamma_{T_r(p^e)}w_{p^r}^{-1}=\Gamma_0(p^{2r}),
\end{align*}
that, by \Cref{rem:isomconj}, imply
\begin{align}\label{eq:isomconj}
X_0(p^{2r+1})\cong X_{B_r(p^e)} \qquad \text{and} \qquad X_0(p^{2r})\cong X_{T_r(p^e)}
\end{align}
respectively.

As explained in \cite[Théorème 2 and the discussion below it]{Edix}, the representation theoretic result in \Cref{lem:reprisom}, together with the isomorphisms in \Cref{eq:isomconj}, implies the first isogeny. The argument to prove the second isogeny is the same, but we also need the isogeny $J_0(p^e)\sim \prod_{r=0}^{e}J_0^\new(p^r)^{\sigma_0\left(p^{e-r}\right)}$, where $\sigma_0(m)$ is the number of positive divisors of the integer $m$.
\end{proof}
The analogous statement for $J^+_\ns(p^e)$ and $p$ odd is given in \cite[Theorem 1.2]{ChenPowers}. For jacobians of Cartan curves of composite level we have the following theorem.

\begin{thm}\label{cor:Chen_too_general}
Let $n>1$ be an integer and let $H<\GLn$ be a Cartan or a Cartan-plus subgroup. Then the jacobian of $X_H$ is a quotient of $J_0(n^2)$. More precisely, if $H$ is a Cartan subgroup, we have:
\begin{equation}\label{Cartan-Borel-isogeny-fine}
\mathrm{Jac}(X_H) \sim\prod_{\substack{ c|a^2 \\ d|b }}J_0^\new(cd^2)^{\sigma_0\left(\frac{a^2}{c}\right)},
\end{equation}
where $\sigma_0(m)$ is the number of positive divisors of an integer $m$ and $a,b$ are positive integers such that $n=ab$ and such that $H$ is split at all primes dividing $a$ and non-split at all primes dividing $b$.
\end{thm}
\begin{proof}
Since all the Cartan-plus subgroups contain a Cartan subgroup, we can suppose that $H$ is a Cartan subgroup. If $b=1$, then $X_H(n)\cong X_0(n^2)$. Thus, we suppose that $b>1$.
Let $b=p_1^{e_1}\cdots p_k^{e_k}$ be the prime factorization of $b$ and for each $j =1,\dots, k$, we set $G_j:=\GL(\Z/p_j^{e_j}\Z)$ and $H_j:= C_{\ns}(p_j^{e_j})<G_j$. Moreover we set $G := \GL(\Z/n\Z)$ and $G_\s := \GL(\Z/a\Z)$, and we choose a totally split Cartan subgroup $H_\s<G_\s$.
Chinese Remainder Theorem gives an identification between $G$ and $G_\s\times \prod_{j=1}^k G_j$ sending $H$ to a conjugate of $H_s \times\prod_{j=1}^{k} H_j$. 

Instead of working with $G$-representations up to isomorphism, it is easier to work inside the representation ring of $G$, namely the Grothendieck ring of the category of finite-dimensional $G$-representations, where we can take differences of representations.
By \Cref{lem:reprisom} we have the following equality in the representation ring of $G_j$ over $\Q$:
\begin{equation*}
\Q\big[G_j/H_j\big] =  \Q\big[ G_j/K_j(p_j^{2e_j}) \big] + 2 \sum_{i=0}^{2e_j-1} (-1)^{i} \Q\big[ G_j/K_j(p_j^{i}) \big],
\end{equation*}
where $K_j(p_j^{2r}):=T_{r}(p_j^{e_j})$, for $r=0,\ldots,e_j$, and $K_j(p_j^{2r+1}):=B_{r}(p_j^{e_j})$, for $r=0,\ldots,e_j{-}1$. Interpreting $G_j$-representations as $G$-representations via the reduction modulo $p_j^{e_j}$ map, the above equality also holds in the representation ring of $G$ over $\Q$. We now get information about the representation $\Q[G/H]$ by taking the tensor product of the above identities, for $j=1,\ldots,k$, and using that, for all the groups $\mathcal G_1, \mathcal G_2$ and all the subgroups $\mathcal H_i < \mathcal G_i$, we have the isomorphisms of $(\mathcal G_1 {\times} \mathcal G_2)$-representations 
\[
\Q[\mathcal G_1/\mathcal H_1] \otimes \Q[\mathcal G_2/\mathcal H_2] \cong \Q[(\mathcal G_1{\times} \mathcal G_2)/(\mathcal H_1 {\times} \mathcal H_2)] .
\]
Denoting by $\otimes$ the product in the representation ring of $G$ over $\Q$, we have 
\begin{equation}\label{eq:repr_ring}
\begin{aligned}
\Q\big[G/H\big] &= \Q\big[G_\s/H_s\big] \otimes \bigotimes_{j=1}^k \Q\big[G_j/ H_j\big]= \\ 
& =  \Q\big[G_\s/H_s\big] \otimes \bigotimes_{j=1}^k \left(  \Q\big[ G_j/K_j(p_j^{2e_j}) \big] + 2 \sum_{i=0}^{2e_j-1} (-1)^{i} \Q\big[ G_j/K_j(p_j^{i}) \big] \right)= \\
& = \sum_{d\mid b^2} \varepsilon(d) {m(d)} \Q\big[G/K(d)\big] ,
\end{aligned}
\end{equation}
where, for every $d= p_1^{f_1} \cdots p_k^{f_k}$ dividing $b^2$, we have
\[
\varepsilon(d) := (-1)^{f_1 + \ldots +f_k}, \quad m(d) := 2^{\# \{j : f_j\neq 2e_j\}} , \quad K(d) := H_s \times \prod_{j=1}^k K_j(p_j^{f_j}) < G.
\]
As explained in \cite{Edix}, Equation (\ref{eq:repr_ring}) implies the following equality in the Grothendieck group of the category of abelian varieties over $\Q$ up to isogeny:
\[
\mathrm{Jac}(X_H) \sim \prod_{d\mid b^2} \mathrm{Jac}(X_{K(d)})^{\varepsilon(d) m(d)}.
\]
Let 
\[
d_1:=\prod_{\substack{j=1 \\ f_j \text{ even}}}^k p_j^{\frac{f_j}{2}}, \quad d_2:=\prod_{\substack{j=1 \\ f_j \text{ odd}}}^k p_j^{\frac{f_j-1}{2}}, \quad p_o:=\prod_{\substack{j=1 \\ f_j \text{ odd}}}^k p_j.
\]
We have $d_1^2d_2^2p_o=d$ and the elements of $\Gamma_{K(d)}$ are exaclty those of the form $\smt{\alpha}{\beta a d_1 d_2 p_o}{\gamma ad_1d_2}{\delta}$, with $\alpha,\beta,\gamma,\delta\in\Z$ and $\alpha\delta-\beta\gamma a^2d=1$. Hence, $w_{ad_1d_2}\Gamma_{K(d)}w_{ad_1d_2}^{-1}=\Gamma_0(a^2d)$, where $w_{ad_1d_2}:=\smt{0}{-1}{ad_1d_2}{0}$. By \Cref{rem:isomconj}, this gives an isomorphism $X_0(a^2d)\cong X_{K(d)}$ and consequently we have
\[
\mathrm{Jac}(X_H) \sim  \prod_{d\mid b^2} J_0(a^2d)^{\varepsilon(d)m(d)}.
\]
Using $J_0(a^2d)\sim \prod_{m|a^2d}J_0^\new(m)^{\sigma_0\left(\frac{a^2d}{m}\right)}$ and the multiplicativity of the arithmetic functions $\varepsilon(d)$ and $m(d)$, one can compute that
\begin{equation*}
\mathrm{Jac}(X_H) \sim \prod_{d\mid b^2} J_0(a^2d)^{\varepsilon(d)m(d)}\sim\prod_{\substack{ c|a^2 \\ d|b }}J_0^\new(cd^2)^{\sigma_0\left(\frac{a^2}{c}\right)}.
\end{equation*}
Hence, in the Grothendieck group of the category of abelian varieties over $\Q$ up to isogeny, $\mathrm{Jac}(X_H)$ is equal to an abelian subvariety of $J_0(n^2)$. This proves the theorem.
\end{proof}

\begin{rem}
In \cite{ChenPowers}, Chen deals with Cartan curves and Cartan subroups whose level is an odd prime power. The computations in our \nameref{sec:appendix} allow us to extend Theorem~$1.1$  in \cite{ChenPowers}, and therefore all the results contained in the paper, to the cases of level $2^e$, for $e$ a positive integer. Notice that $C_{\textnormal{s}}^+(2^e)$ is different from the normalizer of $C_{\textnormal{s}}(2^e)$, in fact at least $\smt1{2^{e-1}}01$ always belongs to the normalizer but does not belong to $C_{\textnormal{s}}^+(2^e)$. Substituting $C_{\textnormal{s}}^+(p^e)$ with the normalizer of $C_{\textnormal{s}}(p^e)$, Theorem~$1.1$ in \cite{ChenPowers} wouldn't extend to the case of level $2^e$.	
\end{rem}

Now we give a lower bound for the genus of Cartan modular curves: we show that for every $\varepsilon >0$ the genus of a Cartan modular curve of level $n$ big enough is larger than $n^{2-\varepsilon}$.
\begin{prop}\label{prop:genus_bounds}
	Let $n\ge 10^5$ be an integer and let $H<\GLn$ be either a Cartan or a Cartan-plus subgroup. Denoting by $g(\Gamma_H)$ the genus of $X_H$ we have
	\[
	g(\Gamma_H) > 0.01\frac{n^{2 - \frac{0.96}{\log \log n}}}{\log \log n} .
	\]
\end{prop}
\begin{proof}
	Since $\det(H) = (\Z/n\Z)^\times$, then $X_H= \Gamma_H \backslash \ol \HH$. Given a congruence subgroup $\Gamma$ of $\tx{SL}_2(\Z)$ containing $-\tx{Id}$, we denote by $d(\Gamma)$ the index $[\tx{SL}_2(\Z):\Gamma]$. Moreover, we denote by $\varepsilon_\infty(\Gamma)$ the number of cusps of $\Gamma\backslash \ol \HH $ and by $\varepsilon_2(\Gamma)$, respectively $\varepsilon_3(\Gamma)$, the number of elliptic points of period $2$, respectively $3$, of $\Gamma \backslash \ol \HH$. Then, by \cite[Theorem 3.1.1]{DS}, the genus of $\Gamma\backslash \ol\HH$ is 
	\begin{equation}\label{eq:genus_solita}
	g(\Gamma) = 1 + \frac{d(\Gamma)}{12} - \frac{\varepsilon_2(\Gamma)}{4} - \frac{\varepsilon_3(\Gamma)}{3} - \frac{\varepsilon_\infty(\Gamma)}{2}.
	\end{equation}
	The numbers $d(\Gamma),\varepsilon_\infty(\Gamma), \varepsilon_2(\Gamma)$ and $\varepsilon_3(\Gamma)$ are multiplicative with the following meaning: Given two coprime integers $n_1,n_2$ and two congruence subgroups $\Gamma_1, \Gamma_2 < \tx{SL}_2(\Z)$ of level $n_1$ and $n_2$ respectively, both containing $-\tx{Id}$, then
\begin{equation}\label{eq:almost_multiplicativity_genus}
\begin{aligned}
	&d(\Gamma_1 \cap \Gamma_2) = d(\Gamma_1) d(\Gamma_2), & &\varepsilon_\infty(\Gamma_1 \cap \Gamma_2) = \varepsilon_\infty(\Gamma_1)  \varepsilon_\infty(\Gamma_2), \\ 
	&\varepsilon_2(\Gamma_1 \cap \Gamma_2) = \varepsilon_2(\Gamma_1) \varepsilon_2(\Gamma_2), & &\varepsilon_3(\Gamma_1 \cap \Gamma_2) = \varepsilon_3(\Gamma_1) \varepsilon_3(\Gamma_2).
\end{aligned}
\end{equation}
	Let $n = p_1^{e_1}\cdots p_k^{e_k}$ the prime factorization of $n$ and we denote by $H_j$ the reduction of $H$ modulo~$p_j^{e_j}$. Then each $H_j$ is either a Cartan or a Cartan-plus subgroup and, under the isomorphism $\GLn\cong \prod_{j=1}^k\GL(\Z/p_j^{e_j}\Z)$, we have $H \cong \prod_{j=1}^k H_j$ and therefore $\Gamma_H= \bigcap_{j=1}^k \Gamma_{H_j}$.
	Last equation, together with the multiplicativity and (\ref{eq:genus_solita}), implies that we can estimate the genus of $X_H$ estimating the quantities $d(\Gamma_H),\varepsilon_\infty(\Gamma_H), \varepsilon_2(\Gamma_H)$ and $\varepsilon_3(\Gamma_H)$ for $n =p^e$. We write these values in Table~\ref{table1}. The numbers $\varepsilon_\infty(\Gamma_H), \varepsilon_2(\Gamma_H)$ and $\varepsilon_3(\Gamma_H)$ can be computed determining, for each representative $r$ of $\tx{SL}_2(\Z)/\Gamma_H$, the ramification index $[\tx{SL}_2(\Z)_\tau:r^{-1}\Gamma_Hr\cap (\Gamma_H)_\tau]$ at $\tx{SL}_2(\Z)\tau$, for $\tau\in\ol{\HH}$, of the $j$-map $j\colon X_H\to \tx{SL}_2(\Z)\backslash\ol{\HH}$, where $(\Gamma_H)_\tau$ is the stabilizer of $\tau$ in $\Gamma_H$. The only non-trivial $\tau$ to check, i.e., those such that $\tx{SL}_2(\Z)_\tau$ is non-trivial, are $\tau\in\{i,e^{\frac{2\pi i}{3}},\infty\}\cup\Q$. See \cite[Sections 3.7 and 3.8]{DS} and \cite[Section 4.1]{DMS} for the split case and \cite[Proposition 7.10]{BaranClass} for the non-split case. In the proof of  \cite[Proposition 7.10]{BaranClass} there is also an explanation of the general method to carry on these calculations.
\begin{table}[h!]  
\renewcommand*{\arraystretch}{1.2}
\caption{Degree, elliptic points and cusps for prime power levels.}
\begin{center}
\begin{tabular}{c|c|c|c|c}
	\toprule
	$H$  & $d(\Gamma_H)$ & $\varepsilon_2(\Gamma_H)$ & $\varepsilon_3(\Gamma_H)$& $\varepsilon_\infty(\Gamma_H)$ 
	\\ \midrule\midrule 
	$C_{\tx{s}}(p^e)$ &$ p^{2e-1}(p{+}1)$ &$ \begin{array}{l} 2 \quad \tx{if }p\equiv 1 \bmod4 \\0 \quad\tx{if }p\nequiv 1\bmod 4\end{array}$ &$ \begin{array}{l} 2 \quad \tx{if }p\equiv 1 \bmod 3 \\0 \quad \tx{if }p\nequiv 1 \bmod 3 \end{array}$ &$p^{e-1}(p{+}1)$ 
	\\ \hline $C_{\tx{s}}^+(p^e)$ 
	&$\frac{p^{2e-1}(p{+}1)}{2}$ &$\begin{array}{l} 2^{e-1} \qquad\qquad  \tx{if }p=2 \\1{+}\frac{p^{e-1}(p{-}1)}{2} \quad \tx{if }p\equiv 1 \bmod4 \\ \frac{p^{e-1}(p{+}1)}{2} \quad\quad \,\, \tx{if }p\equiv 3 \bmod4\end{array}$ &$ \begin{array}{l}  1\quad \tx{if }p\equiv 1 \bmod 3 \\0 \quad \tx{if }p \nequiv 1 \bmod 3 \end{array} $& $\begin{array}{l} 2 \, \tx{ if }p^e = 2 \\ \frac{p^{e-1}(p{+}1)}{2} \end{array}$
	\\ \hline $C_{\tx{ns}}(p^e)$  
	&$p^{2e-1}(p{-}1)$ & $ \begin{array}{l}0 \quad \tx{if }p\nequiv 3 \bmod4 \\2 \quad\tx{if }p\equiv 3 \bmod 4\end{array} $ &$ \begin{array}{l}0 \quad \tx{if }p\nequiv 2 \bmod3 \\2 \quad\tx{if }p\equiv 2 \bmod 3\end{array}$& $p^{e-1}(p{-}1)  $ 
	\\ \hline $C_{\tx{ns}}^+(p^e)$ 
	&$\frac{p^{2e-1}(p{-}1)}{2}$ & $ \begin{array}{l}  2^{e-1}  \qquad\qquad \,  \tx{if }p=2 \\\frac{p^{e-1}(p{-}1)}{2} \qquad \; \tx{if }p\equiv 1 \bmod4 \\1 {+} \frac{p^{e-1}(p{+}1)}{2} \quad \tx{if }p\equiv 3 \bmod4 \end{array} $ &$ \begin{array}{l}0 \quad \tx{if }p\nequiv 2 \bmod3 \\1 \quad\tx{if }p\equiv 2 \bmod 3\end{array} $& $\begin{array}{l} 1 \,\tx{ if }p^e = 2 \\
	\frac{p^{e-1}(p{-}1)}{2} \end{array}$\\
	\bottomrule
\end{tabular}
\end{center}
\label{table1}
\end{table}
	The table implies that for every prime $p_j$ dividing $n$ with exponent $e_j$ we have
	\[
	d(\Gamma_{H_j}) \geq \tfrac{1}{2} p_j^{2e_j}(1-\tfrac 1{p_j}), \quad \varepsilon_2(\Gamma_{H_j}) \le p_j^{e_j}, \quad \varepsilon_3(\Gamma_{H_j}) \le 2, \quad \varepsilon_\infty(\Gamma_{H_j}) \leq p_j^{e_j}(1+\tfrac 1{p_j}).
	\]
These inequalities and the multiplicativity (\ref{eq:almost_multiplicativity_genus}) imply the following estimates for $n\ge 15$:
\[
\begin{aligned}
d(\Gamma_{H}) & \geq \frac{n \phi(n)}{2^{\omega(n)}} >  \frac{n^2}{4.4 \, \log \log(n) 2^{\omega(n)}} \ge  \frac{n^2}{4.4 \, \log \log(n)2^{1.3841 \frac{\log n}{\log \log n}}}  > \frac{n^{2-\frac{0.96}{\log\log n}}}{4.4 \, \log \log n } ,
\\  \varepsilon_2(\Gamma_{H}) & \le n , \quad  \varepsilon_3(\Gamma_{H}) \le 2^{\omega(n)} \le n , \quad \varepsilon_\infty(\Gamma_{H}) \leq n \prod_{j=1}^k (1 + \tfrac{1}{p_j}) \le \sigma_1(n) \leq 2.59 n \log\log n ,
\end{aligned}
\]
where $\phi(n)$ is Euler's totient function which is estimated using \cite[Theorem $15$]{RosSch}, $\omega(n)=k$ is the number of prime divisors of $n$ which is estimated as in \cite[Théorème 11]{AnalOmega}, and $\sigma_1(n)$ is the sum of positive divisors of $n$ which is estimated as in \cite[Theorem 1]{AnalSigma1}. For $n\ge 10^5$, substituting in (\ref{eq:genus_solita}), we get
\[
\begin{aligned}
g(\Gamma_H)  & > 1 + \frac{n^{2- \frac{0.96}{\log\log n}}}{52.8 \, \log \log n} - \frac n 3 - \frac n 4 - 1.3 n \log \log n  \ge 0.01\frac{n^{2 - \frac{0.96}{\log \log n}}}{\log \log n} .
\end{aligned}
\]
\end{proof}

\section{Field of definition of automorphisms}\label{Sec:FieldDef}
In this section we prove that,  when the level is large enough, every automorphism of the modular curve $X_H$ associated to a subgroup $H$ of $\GLn$ is defined over the compositum of some quadratic fields, and in some cases we find explicitly this field.

Whenever $K$ is a field, $X$ is a variety over $K$, and $F$ is an extension of $K$, we write $\Aut_F(X)$ for the set of automorphisms  of $X$ defined over $F$; analogously we use the notations $\End_F(X)$ and $\tx{Hom}_F(X,Y)$ for $X$ and $Y$ being abelian varieties over $K$. Whenever we omit the dependency on the field, we mean automorphisms (or endomorphisms) defined over the algebraic closure of $K$; in particular when $X$ is a modular curve the ``group of the automorphisms of X'' is $\Aut_{\ol \Q}(X)$ or equivalently $\Aut_\C(X)$. 
We start with a straightforward generalization of \cite[Lemma 1.4]{KM}.

\begin{lem}\label{JacFactorEnd}
Let $K$ be a perfect field with algebraic closure $\ol K$, let $X$ be a smooth projective and geometrically connected curve defined over $K$ of genus $g(X)$ and let $\tx{Jac}(X)$ be its jacobian variety.
We suppose that there are two abelian varieties $A_1$ and $A_2$ over $K$ such that $\tx{Hom}_{\ol K}(A_1,A_2)=0$ and such that $\tx{Jac}(X)$ is isogenous over $K$ to $A_1{\times_K} A_2$. 
If 
$$g(X)>2 \dim(A_2)+1,$$
and if $F\subset \ol K$ is an extension of $K$ such that $\End_{\ol K}(A_1)=\End_F(A_1) $, then every automorphism of $X$ over $\ol K$ can be defined over $F$.
\end{lem}
\begin{proof}
We fix isogenies $\varphi\colon \tx{Jac}(X) \to A_1{\times_K}A_2$ and $\varphi^\vee \colon A_1{\times_K}A_2\to \tx{Jac}(X)$ whose compositions are multiplications by an integer.
Let $u\in\Aut_{\ol K}(X)$ and $\sigma \in \text{Gal}(\overline{K}/F)$ and consider the automorphism $v:= u^\sigma \circ u^{-1}$. Let $Y$ be the quotient of $X$ by the subgroup of automorphisms generated by $v$ (which is finite since $g(X)\ge 2$) and let $\tx{Jac}(Y)$ be the jacobian of $Y$. Because of $\varphi$, and since $\tx{Hom}_{\ol K}(A_1,A_2)=0$, we can identify $u_*, u_*^\sigma \in \Aut_{\ol K}(\tx{Jac}(X))$ respectively with
\[
(u_1,u_2), (u_1^\sigma, u_2^\sigma)  \in \left(\End_{\ol K}(A_1) \otimes\Q\right)^\times\times \left(\End_{\ol K}(A_2)\otimes \Q\right)^\times \cong \left(\End_{\ol K}(A_1{\times_K}A_2) \otimes \Q \right)^\times.
\] 
Since $\End_{\ol K}(A_1)=\End_F(A_1)$, then $u_1=u_1^\sigma$, and $v_*=(\tx{id},v_2)$. 

This implies that
\begin{equation*}\label{ineq1_gY}
\dim(A_1) \le \dim((A_1 \times_K A_2)^{(\tx{id},v_2)})\le \dim(\tx{Jac}(X)^v),
\end{equation*}
where $\tx{Jac}(X)^v$ is the $v$-invariant subvariety of $\tx{Jac}(X)$. A point of $\tx{Jac}(X_{\ol K})^v$ is associated to a divisor $D$ on $X_{\ol K}$ which is equivalent to $vD$ in $\tx{Jac}(X_{\ol K})$. This implies, using the cyclicity of $\langle v\rangle$ and \cite[Corollary 7]{GoldsteinPic}, that every element of $\tx{Jac}(X_{\ol K})^v$ is represented by a $v$-invariant divisor on $X_{\ol K}$. This means that $\tx{Jac}(X_{\ol K})^v$ is the image of the pullback map $\pi^*\colon\tx{Jac}(Y_{\ol K})\rightarrow\tx{Jac}(X_{\ol K})$. So we have that
\begin{equation*}\label{ineq2_gY}
\dim(\tx{Jac}(X)^v)=\dim(\tx{Jac}(X_{\ol K})^v) = \dim(\tx{Jac}(Y_{\ol K})) = g(Y),
\end{equation*}
where $g(Y)$ is the genus of $Y$, and, therefore, we have
\[
g(X) - \dim(A_2) = \dim(A_1) \le g(Y).
\]
Hence, by the Riemann-Hurwitz formula applied to the projection $X\to Y$, we have
$$\text{dim}(A_1)+\text{dim}(A_2)-1\ge d(g(Y) -1)\ge d \dim(A_1)-d,$$
where $d$ is the order of $v$. If $d>1$, we get $\text{dim}(A_1)\le \text{dim}(A_2)+1$, which is impossible by hypothesis. Hence $d=1$ and $v$ is the identity. This implies that $u^\sigma=u$, for every $\sigma\in\text{Gal}(\overline{K}/F)$, i.e., since $K$ is perfect, $u\in\Aut_F(X)$.
\end{proof}
For every abelian variety $A$ over a number field $K$, let $A^\CM$ be the maximal abelian subvariety of $A_{\ol K}$ that is isogenous to a product of simple CM abelian varieties and let $A^\tx{N}$ be the maximal abelian subvariety of $A_{\ol K}$ that is isogenous to a product of simple non-CM abelian varieties. We call $A^\CM$ the CM-part of $A$ and $A^\tx{N}$ the non-CM-part of $A$. Both the CM part and the non-CM part of $A$ are also defined $K$, since  by definition we have $(A^\CM)^\tau = A^\CM$ and $(A^\tx{N})^\tau = A^\tx{N}$ for each $\tau \in \Gal(\ol K/K)$. The dimension of the CM-part and the dimension of the non-CM part are invariant under isogeny on $A$. Hence, by looking at the decomposition of $A$ in simple factors, we see that $\dim (A^\CM)+\dim(A^\tx{N})=\dim(A)$.

 We want to apply Lemma \ref{JacFactorEnd} to the case $A_1=\mathrm{Jac}(X)^{\tx N}$ and $A_2=\mathrm{Jac}(X)^{\CM}$. Hence, we are interested in an  upper bound on the dimension of the CM part of the jacobian of Cartan modular curves. By \Cref{cor:Chen_too_general}, it is enough to know an upper bound in the case $X=X_0(n)$. 
\begin{prop}\label{prop:g^c_0(N)}
For every integer $n>1$, the dimension $g_0^\CM(n)$ of the CM part of $J_0(n)$ satisfies 
\[
g_0^\CM(n) \leq 9 \log (n)^2 n^{\frac 1 2 + \frac{2.816}{\log \log n}}. 
\]
\end{prop}
\begin{proof}
For every positive integer $k$, let $J_0^\new(k)$ be the new part of $J_0(k)$ and let $\sigma_0(k)$ be the number of positive divisors of $k$. Then we have a canonical isogeny
\[
J_0(n) \sim \prod_{d|n} J_0^\new(d)^{\sigma_0(n/d)}.
\]
Denoting by $g_0^{\new, \CM}(d)$ the dimension of the CM part of $J_0^\new(d)$, we also have
\begin{equation}\label{eq:g_0^c_equal_sum_new_parts}
g_0^\CM(n) = \sum_{d|n} \sigma_0(n/d)  g_0^{\new, \CM}(d).
\end{equation}
We know that $J_0^\new(d)$ is isogenous over $\Q$ to $\prod_{[f]}A_{f}$, where $[f]$ is the Galois orbit of the newform $f$, and the cardinality of $[f]$ is equal to the dimension of $A_f$ (see \cite[Chapter~6]{DS}). We now look at the CM part of each factor $A_f$.

In order to do this, we describe a bijection between the set of normalized newforms contributing to the CM part and a suitable set of triples $(K, \mathfrak m, \lambda)$, where $K$ is an imaginary quadratic field, $\mathfrak m$ is an ideal of the ring of integers $\calO_K$ of $K$ and $\lambda$ is a primitive Gr\"ossencharacter of $K$ defined modulo $\mathfrak m$ (see \cite[Section~4]{ShimuraCM} for the definition of primitive Gr\"ossencharacter). Namely, let $(K, \mathfrak m, \lambda)$ be a triple as above such that the nebentypus associated to $\lambda$ is trivial and let $\Delta_K$ be the discriminant of $K$, then, by \cite[Lemma~3 and Theorem~1]{ShimuraCM}, we can construct a normalized newform $f_\lambda$ of level $|\Delta_K||\mathfrak m|$ such that the associated abelian variety $A_{f_\lambda}$ is isogenous over $\C$ to a product of elliptic curves with CM over $K$. On the other hand, by \cite[Proposition~1.6]{ShimuraClassFields}, for each normalized newform $f$, if the associated abelian variety $A_{f}$ has non-trivial CM part, then $f=f_\lambda$ for a unique triple $(K, \mathfrak m, \lambda)$ as above. 
This gives a bijection between the set of normalized newforms $f$ contributing to the CM part of  $J_0^\new(d)$ and the set of triples $(K, \mathfrak m, \lambda)$ described above such that $|\Delta_K||\mathfrak m|=d$.
By \cite[Proposition~1.6]{ShimuraClassFields}, an abelian variety $A_f$ has non-trivial CM part if and only if it is isogenous over $\C$ to a product of elliptic curves with CM over $K$. Hence the number of normalized newforms $f$ of level $d$ contributing to the CM part is equal to the dimension of the CM part of $J_0^\new(d)$.
In conclusion, $g_0^{\new,\CM}(d)$ is equal to the number of triples $(K, \mathfrak m, \lambda)$ where $\lambda$ is a primitive Gr\"ossencharacter of $K$ defined modulo $\mathfrak m$  with trivial nebentypus and $|\Delta_K||\mathfrak m|=d$.

We now give an upper bound on the number of such triples. For every choice of $K$ and $\mathfrak m$, the set of primitive Gr\"ossencharacters of $K$ defined modulo $\mathfrak m$ is a subset of the set of Gr\"ossencharacters of $K$ defined modulo $\mathfrak m$. If this set is not empty, then there is at least one Gr\"ossencharacter $\lambda_0$ and all other Gr\"ossencharacters are given by $\lambda_0  \chi$, for $\chi$ a character of the group
\[
\ClZ(K) := \frac{\{\tx{fractional ideals of } \calO_K \tx{ coprime to } \mathfrak m \}}{\{(\alpha) : \exists a \in \Z \tx{ coprime to } \mathfrak m \tx{ such that } \alpha \equiv a \bmod{\mathfrak m} \}} .
\]
Thus, for given $K$ and $\mathfrak m$, the cardinality of $\ClZ(K)$ is larger than the number of triples $(K, \mathfrak m, \lambda)$ we are interested in, hence
\begin{equation}\label{eq:g^newc_and_ClZ}
g_0^{\new, \CM}(d) \leq \sum_{|\Delta_K| |\mathfrak m| = d} \# \ClZ(K)\,.
\end{equation}
To give a bound on $\ClZ(K)$ we look at the following short exact sequence
\[
1 \lto \frac{(\calO_K/\mathfrak m)^\times}{\calO_K^\times \cdot (\Z/(\Z{\cap}\mathfrak m))^\times} \lto \ClZ(K) \lto \tx{Cl}(K) \lto 0,
\]
where $\tx{Cl}(K)$ is the class group of $K$ and we write $\calO_K^\times $ and  $(\Z/(\Z{\cap}\mathfrak m))^\times$ in place of their natural image inside $(\calO_K/\mathfrak m)^\times$. We write $\mathfrak m = \prod_{p} \mathfrak m_p$ for $p$ varying in the set of rational primes and $\mathfrak m_p$ being a product of primes of $\calO_K$ dividing $p$. Thus the above short exact sequence gives 
\begin{align*}%\label{eq:ClZ_and_factorization}
\# \ClZ(K) & \leq \# \tx{Cl}(K) \! \cdot \!\#\! \left(\frac{(\calO_K/\mathfrak m)^\times}{(\Z/(\Z{\cap}\mathfrak m))^\times}\right) \!  =\!   \#\tx{Cl}(K)  \prod_{p \,\mid\, |\mathfrak m|} \! \#\! \left(\frac{(\calO_K/\mathfrak m_p)^\times}{
(\Z/(\Z{\cap}\mathfrak m_p))^\times}\right) \le \\
& \le 3 \log(|\Delta_K|) \sqrt{|\Delta_K|} \prod_{p\,\mid\, |\mathfrak m|} \! \left( (1 + \tfrac 1 p)|\mathfrak m_p|^{1/2} \right)\!  = \!  3 \log(|\Delta_K|) \sqrt{|\Delta_K||\mathfrak m|} \prod_{p \,\mid\, |\mathfrak m|} \! (1 + \tfrac 1 p),
\end{align*}
where the class number of $K$ is estimated using \cite[Theorem $8.10$ and Lemma $8.16$]{Nark} and the bound on the cardinality of ${(\calO_K/\mathfrak m_p)^\times}/{(\Z/(\Z{\cap}\mathfrak m_p))^\times}$ is trivial after factoring $\mathfrak m_p$.
Substituting in (\ref{eq:g^newc_and_ClZ}), we have
\begin{align*}
g_0^{\new, \CM}(d) & \leq \sum_{|\Delta_K| |\mathfrak m| = d} \Big(3\sqrt d  \log(|\Delta_K|) \prod_{p|\, |\mathfrak m|} (1 + \tfrac 1 p)\Big).
\end{align*}
Let $M_d:= \# \Big\{(K,\mathfrak m): |\Delta_K| |\mathfrak m| = d \Big\}$ and for $m\in \Z_{\ge1}$, we denote by $\sigma_1(m)$ the sum of the positive divisors of $m$. We have $\sigma_1(m)<3m \log m$, for each $m\ge 2$ (see \cite[Theorem~1]{AnalSigma1} if $m\ge 7$, it is trivial in the remaining cases). Then 
\begin{align*}
g_0^{\new, \CM}(d) & \leq 3M_d \sqrt{d} \log (d) \prod_{p|d} (1 + \tfrac 1 p)  \leq 3M_d \sqrt{d} \log (d) \tfrac{\sigma_1(d)}{d} \le 9 M_d \sqrt{d} \log(d)^2 . %qua potrebbe esserci log(d) loglog(d) ma poi ci sono problemi per d=2...
\end{align*}
Substituting in (\ref{eq:g_0^c_equal_sum_new_parts}), we get
\begin{equation}\label{eq:g_0^c_first_est}
\begin{aligned}
g_0^\CM(n) &\leq 9\sum_{d|n} \sigma_0(n/d)M_d\sqrt{d} \log (d)^2  \le 9\sqrt{n} \log (n)^2\sum_{d|n} M_d\sigma_0(n/d)  \le \\
& \leq 9 \sqrt n \log(n)^2   \# \Big\{(K,\mathfrak m, d): |\Delta_K| |\mathfrak m| d \tx{ divides } n \Big\}.
\end{aligned}
\end{equation}
Writing the prime factorization $n = \prod_{i=1}^r p_i^{e_i}$, we know that an imaginary quadratic field $K$ with discriminant dividing $n$ must be $K = \Q(\sqrt{- \prod_{i=1}^r p_i^{\varepsilon_i}})$, with $\varepsilon \in \{0,1\}^r$. Hence
\begin{align*}
&\# \Big\{(K,\mathfrak m, d): |\Delta_K| |\mathfrak m| d \tx{ divides } n \Big\} \leq  \sum_{\varepsilon \in \{0,1\}^r} \#\Big\{ (\mathfrak m, d): |\Delta_K| |\mathfrak m| d\tx{ divides } n  \Big\} \le \\
& \leq \sum_{\substack{\varepsilon \in \{0,1\}^r\\ m \in \Z_{>0}}}  \#\Big\{ \mathfrak m \subset \calO_K: |\mathfrak m|=m\Big \}\cdot  \#\Big\{d \in \Z_{>0}: dm\prod_{i=1}^r p_i^{\varepsilon_i} \tx{ divides }n \Big\}.
\end{align*}
We have the factorizations $m=\prod_{i=1}^r p_i^{f_i}$ and $d = \prod_{i=1}^r p_i^{c_i}$, where $f_i,c_i\in\{0,1,\ldots,e_i\}$, for $i=1,\ldots,r$, and we denote by $f$ the $r$-tuple whose components are the $f_i$'s and similarly we define $c$.
Then the number of ideals $\mathfrak m$ in $\calO_K$ having norm $m$ is lesser than $\prod_{i=1}^r(f_i{+}1)$ which is equal to the number of pairs $(a,b)$ of elements of $\Z_{\geq 0}^r$ such that $a + b = f$. Hence we get
\begin{align*}
&\# \Big\{(K,\mathfrak m, d){:} |\Delta_K| |\mathfrak m| d \tx{ divides } n \Big\} {\le} \# \Big\{ (\varepsilon, a,b,c) \in \{0,1\}^r {\times} (\Z_{\geq 0}^r)^3 \!\!: \varepsilon_i + a_i {+} b_i {+} c_i \leq e_i \Big\}{\le} \\
& {\le}\prod_{i=1}^r \Big(\# \Big\{ (a_i,b_i,c_i) \in \Z_{\geq 0}^3 : a_i {+} b_i {+} c_i \leq e_i \Big\} + \# \Big\{ (a_i,b_i,c_i) \in \Z_{\geq 0}^3 : a_i {+} b_i {+} c_i \le e_i{-}1 \Big\} \Big) {\le} \\
&  {\le} \prod_{i=1}^{r} \Big( \binom{e_i+3}{3} + \binom{e_i+2}{3} \Big) {\le} \prod_{i=1}^r \frac{(e_i+2)(e_i+1)^2}{2}.
\end{align*}
Notice that $\sigma_0(n)=\prod_{i=1}^r (e_i{+}1)$ is the number of positive divisors of $n$ and $ \prod_{i=1}^r \frac{(e_i{+}2)(e_i{+}1)}{2}$ is the number of triples $(d_1,d_2,d_3)$ of positive integers such that $d_1d_2d_3 = n$. Using the upper bounds, contained in \cite{AnalSigma0} and \cite{AnalTau3}, for these two quantities, we get
\[
\# \Big\{(K,\mathfrak m, d): |\Delta_K| |\mathfrak m| d \tx{ divides } n \Big\} \leq n^{\frac{1.538 \log 2}{\log\log n}} n^{\frac{1.592 \log 3}{\log \log n}} \leq n^{\frac{2.816}{\log \log n}}.
\]
Substituting in (\ref{eq:g_0^c_first_est}) we find
\[
g_0^\CM(n) \leq  9 \sqrt n \log (n)^2 n^{\frac{2.816}{\log \log n}} =  9  \log (n)^2 n^{\frac 1 2 + \frac{2.816}{\log \log n}}.
\]
\end{proof}
When the level is a prime power, the previous upper bound is easier and smaller.
\begin{prop}\label{prop:g^c_0(p^e)} 
For every prime $p$ and positive integer $e$, the dimension $g_0^\CM(p^e)$ of the CM part of $J_0(p^e)$ satisfies 
\[
g_0^\CM(p^e) \le \begin{cases}
13 \, \sqrt{2^{e}} & \quad \tx{ if }p=2, \\
0 & \quad \tx{ if }p\equiv 1\bmod 4, \\
5.5 \, \sqrt{p^e}\log p & \quad \tx{ if }p\equiv 3\bmod 4.
\end{cases} 
\]
\end{prop}
The proof follows the same steps of the previous proposition and is simplified by the fact that there are few quadratic imaginary fields $K$ whose discriminant divides $p^e$. More precisely: there are two fields when $p=2$, there are no fields if $p \equiv 1 \bmod 4$ and there is only one field if $p \equiv 3 \bmod 4$. We now give an upper bound for the field of definition of the automorphisms of a Cartan modular curve of large enough level.
\begin{prop}\label{prop:field_Cartan(n)}
Let $n\ge 10^{400}$ be an integer and let $H<\GLn$ be either a Cartan or a Cartan-plus subgroup. Then every automorphism of $X_H$ is defined over the compositum of all the quadratic fields whose discriminant divides~$n$.
\end{prop}
\begin{proof}
Let $J_H$ be the jacobian of $X_H$ and let $J_H^\CM$ and $J_H^{\tx N}$ be the CM part and the non-CM part of $J_H$ respectively. 
By Lemma \ref{JacFactorEnd}, it is enough to prove that $2\dim(J_H^\CM){+}1$ is smaller than the genus of $X_H$ and that every endomorphism of $J_H^{\tx{N}}$ is defined over the compositum of all the quadratic fields whose discriminant divides $n$. The latter is true because, by Theorem \ref{cor:Chen_too_general}, $J_H^\tx{N}$ is a quotient of $J_0(n^2)^{\tx N}$ and by in \cite[Proposition 1.3]{KM} every endomorphism of $J_0(n^2)^{\tx N}$ is defined over the compositum of all the quadratic fields whose discriminant divides $n$. 
By Theorem \ref{cor:Chen_too_general} $J_H^\CM$ is a quotient of $J_0(n^2)^\CM$ hence we can use Proposition \ref{prop:g^c_0(N)} to bound the $\dim(J_H^\CM)$; this, together with the bound for the genus $g(X_H)$ of $X_H$ given in Proposition \ref{prop:genus_bounds}, implies the inequality we need when $n\ge 10^{400}$:
\[
2 \dim(J_H^\CM) +1  \le 2 \dim(J_0(n^2)^\CM)+1  \le 73 \log(n)^2 n^{1 + \frac{5.632}{\log \log n}} < \frac{n^{2 - \frac{0.96}{\log \log n}}}{100 \log \log n} < g(X_H).
\]
\end{proof}

Proposition \ref{prop:field_Cartan(n)} can be made sharper when $n$ is a prime power.
\begin{prop}\label{prop:field_Cartan(p^e)}
Let $p$ be a prime and $e$ a positive integer and let $X$ be a curve associated to a Cartan or a Cartan-plus subgroup of level $p^e$. If the genus of $X$ is at least $2$, then every automorphism of $X$ is defined over the field
\[
K_p = 
\begin{cases} 
\Q(i, \sqrt 2),  &\quad \tx{if }p=2, \\
\Q\left( \sqrt{p} \right), &\quad \tx{if }p\equiv 1 \bmod 4, \\
\Q\left( \sqrt{-p} \right), &\quad \tx{if }p\equiv 3 \bmod 4.
\end{cases}
\] 
\end{prop}
A strategy of proof is the same of Proposition \ref{prop:field_Cartan(n)}:
\begin{enumerate}[(i)]
\item\label{item1_rom} give an upper bound for $\dim (\tx{Jac}(X)^\CM)$;
\item\label{item2_rom} give a lower bound for the genus;
\item\label{item3_rom} apply \cite[Proposition 1.3]{KM} and Theorem \ref{cor:Chen_too_general} to deduce that the endomorphisms of $\tx{Jac}(X)^{\tx{N}}$ are defined over $K_p$;
\item\label{item4_rom} apply Lemma \ref{JacFactorEnd}.
\end{enumerate} 
In particular in the case of $X_\ns(p^e)$ and $X_\ns^+(p^e)$, when $p^e > 600$, the propositions \ref{prop:Chen_Xns(p^e)} and \ref{prop:g^c_0(p^e)} and Table \ref{table1} give bounds in (\ref{item1_rom})  and (\ref{item2_rom}) that are sharp enough for (\ref{item4_rom}). If $p^e \le 600$, the bounds in Proposition \ref{prop:g^c_0(p^e)} are sometimes not sharp enough. In these cases we can compute explicitly the CM part and notice that only a factor of it of low dimension has endomorphisms defined over a field bigger than $K_p$: whenever a CM factor is a rational elliptic curve, we know by CM theory that its endomorphisms are defined over $K_p$ and it can be discarded from the count. This is done in the MAGMA script available at \cite{Script_aut}. 
The case $X_\text{s}(p^e) \cong X_0(p^{2e})$ follows from \cite[Corollary $1.14$]{KM} and the case $X_\tx{s}^+(p^e)\cong X_0(p^{2e})$ follows from the following proposition.

\begin{prop}\label{prop:field_X0+(p^e)}
	Let $p$ be a prime and $e$ a positive integer. If the genus of $X_0^*(p^e)$ is at least $2$, then every automorphism of $X_0^*(p^e)$ is defined over the field
	\[
	K_p = \begin{cases} 
	\Q(i, \sqrt 2),  &\quad \tx{if }p=2 , \\
	\Q\left( \sqrt{p} \right), &\quad \tx{if }p\equiv 1 \bmod 4, \\
	\Q\left( \sqrt{-p} \right), &\quad \tx{if }p\equiv 3 \bmod 4.
	\end{cases}
	\] 
\end{prop}
Again, one can apply the same strategy used for Propositions \ref{prop:field_Cartan(n)} and \ref{prop:field_Cartan(p^e)}, together with the MAGMA script available at \cite{Script_aut}. In particular we need a lower bound for the genus of $X_0^*(p^e)$. Since we do not know an explicit reference giving a formula for this genus, we write it in the following remark.

\begin{rem}\label{rem:g0+}
	Given a positive integer $n$, let $X_0^+(n)$ be the quotient of $X_0(n)$ by the $n$-th Atkin-Lehner operator. This curve is equal to $X_0^*(n)$ when $n$ is the power of a prime. 

In \cite[Equation $9$]{Ogg75} there is a formula for the genus $g_0^+(n)$ of $X_0^+(n)$ when $n$ is prime. When $n=p^{2e}$ with $p$ prime, we can compute $g_0^+(n)$ using Table \ref{table1} since $X_0^+(n)$ is isomorphic to a split Cartan curve. For general $n$, \cite[Equation $9$]{Ogg75} can be easily generalized applying Riemann-Hurwitz formula to the natural map $X_0(n) \to X_0^+(n)$  and counting the number of fixed points of the $n$-th Atkin-Lehner operator. This gives
	\[
	g_0^+(n) = \begin{cases}
	0, \quad & \tx{if }n \in \{1,2,3,4\}, \\
	\frac{1 + g_0(n)}{2} - \frac{h(-n) + h(-4n)}{4}, \quad & \tx{if $n\ge5$ is odd},\\ 
	\frac{1 + g_0(n)}{2} - \frac{h(-4n)}{4}, \quad & \tx{if $n\ge5$ is even},
	\end{cases}
	\]
	where $g_0(n)$ is the genus of $X_0(n)$ and $h(D)$ is the class number of the quadratic order with discriminant $D$, with the convention $h(D)=0$ if $D$ is a square or if $D \equiv 2,3 \bmod 4$.
\end{rem}

\begin{rem}\label{rem:48}
We are not always able to prove that every automorphism of a Cartan modular curve is defined over a compositum of quadratic fields. For example, an analogous of \Cref{Cartan-Borel-isogeny-fine} for Cartan-plus curves, proved using Chen's isogeny in \cite{ChenPowers}, implies that the jacobian of the totally non-split Cartan-plus curve $X$ of level $48$ contains  $J_0^{\new,*}(48^2)$. Since there are two CM (weight 2) newforms of level $48^2$ of degree $2$ and invariant under the action of both the Atkin-Lehner operators $w_9$ and $w_{256}$, then the jacobian $J_0^{\new,*}(48^2)$ has a CM part of dimension at least $4$ whose endomorphisms could be defined over a field bigger than the compositum of quadratic fields. This prevents us from applying \Cref{JacFactorEnd} in (\ref{item4_rom}) of the strategy above, because the genus of $X$ is $9$ (see \Cref{table1}).
\end{rem}

\section{Automorphisms}\label{Sec:Aut}

In this section we treat our main problem, namely to determine the automorphisms of certain modular curves $X_H$ over $\C$ for a subgroup $H$ of $\GLn$. We restrict our attention to $X_H$ geometrically connected, i.e., $\det(H) = (\Z/n\Z)^\times$. Every automorphism we are interested in induces an automorphism of the Riemann surface $X_H(\C)= \Gamma_H\backslash \ol{\HH}$ and, since it is compact, each of these automorphisms comes from an automorphism of the algebraic curve $(X_H)_\C$. Let $\PP \colon \tx{GL}_2^+(\Q)\to \tx{PGL}_2^+(\Q)$ be the natural map. Each matrix $m \in \tx{PGL}^+_2(\Q)$ defines a M\"obius transformation $m\colon \ol\HH \to \ol\HH$ and such an automorphism of the Riemann surface $\ol\HH$ pushes down to an automorphism of $\Gamma_H\backslash \ol{\HH}$ if and only if $m$ normalizes $\PP(\Gamma_H)$.

\begin{defi}
Let $H$ be a subgroup of $\GLn$ such that $\det(H)= (\Z/n\Z)^\times$. An automorphism of $X_H$ defined over $\C$ is \emph{modular} if its action on $X_H(\C) = \Gamma_H\backslash \ol{\HH}$ is described by a M\"obius transformation associated to a matrix $m \in \tx{PGL}^+_2(\Q)$ normalizing $\PP(\Gamma_H)$.
\end{defi}
When $H$ has surjective determinant, $\Aut(X_H)$ contains the subgroup of modular automorphisms which is isomorphic to $\mathcal N/\PP(\Gamma_H)$, where $\mathcal N$ is the normalizer of $\PP(\Gamma_H)$ inside $\tx{PGL}_2^+(\Q)$. 

\begin{rem}\label{PGL2Q-R}
Notice that we can define modular automorphisms of $Y_H$ looking at $\tx{PGL}_2^+(\R)$, instead of $\tx{PGL}_2^+(\Q)$, as follows: an automorphism $\iota$ of $Y_H(\C)=\Gamma_H\backslash \HH$ is \emph{modular} if there is a matrix $m \in \tx{PGL}_2^+(\R)$ that normalizes the image of $\Gamma_H$ in $\tx{PGL}_2^+(\R)$ and hence defines a M\"obius transformation $m\colon \HH \to \HH$ that pushes down to $\iota$. This is equivalent to the previous definition. Indeed if $\tilde m \in \tx{GL}_2^+(\R)$ is a lift of $m$, then $\tilde m$ normalizes $\Gamma_{\pm H} = (\R^\times\Gamma_H)\cap \tx{SL}_2(\R)$, hence conjugation by $\tilde m$ preserves the set of $\Q$-linear combinations of matrices in $\Gamma_{\pm H}$, which is equal to the set of matrices with entries in $\Q$. Looking at the conjugates by $\tilde m$ of the matrices $\smt 1000$, $\smt 0100$, $\smt 0010$ and $\smt 0001$, we easily deduce that $\tilde m$ is a real multiple of a matrix in $\GL(\Q)$, and consequently $m$ lies in $\tx{PGL}^+_2(\Q)$. 

In other words: every modular automorphism of $Y_H(\C)$ extends to a modular automorphism of $X_H$ and, conversely,  every modular automorphism of $X_H$ preserves the set of cusps, hence restricts to a modular automorphism of $Y_H(\C)$.
\end{rem}

If an automorphism is modular, then it preserves the set of cusps and the set of branch points for the map $\HH \to \Gamma_H\backslash\HH$, because the map $\ol\HH \to X_H(\C)$ is branched on these sets and $\PP^1(\Q)$ is stable under $\tx{PGL}_2(\Q)$.
The converse is also true, as shown in \cite[Proposition 3.1]{DoseAut}. In the following lemma we use this criterion to give a (different) sufficient condition for an automorphism to be modular.

\begin{lem}\label{lem:aut_general_H}		
Let $n$ be a positive integer, let $H$ be a subgroup of $\GLn$ containing the scalar matrices and such that $\det(H) = (\Z/n\Z)^\times$, and let $\tx{gon}(X_H)$ be the gonality of $X_H$. If there is a prime $\ell$ not dividing $n$ such that $5\le \ell < \tfrac 1 2 \tx{gon}(X_H)-1$, then every automorphism of $X_H$ defined over a compositum of quadratic fields is modular.
\end{lem}
\begin{proof}
Let $u$ be an automorphism of $X_H$ defined over the compositum of some quadratic fields and let $P\in X_H(\C)$ be either a cusp or a branch point of the map $\HH \to \Gamma_H\backslash\HH= Y_H(\C)$. By \Cref{PropRamPoint}, if $P$ is branch point, then we have either $P=(E_i, \phi)$ with $\phi^{-1} {\circ} i|_{E_i[n]} {\circ} \phi\in \pm H$ or $P=(E_\rho,\phi)$ with $\phi^{-1} {\circ} \rho|_{E_\rho[n]} {\circ} \phi\in \pm H$. Hence, in both cases (cusp or branch point) we can apply \Cref{theo:repetitions} and deduce which particular multiplicities appear in $T_\ell (P)$: depending on the case, $T_\ell (P)$ contains either a point with multiplicity at least~$4$ or a point with multiplicity~$3$ or two distinct points with multiplicity~$2$. Since $u$ is defined over a compositum of quadratic fields and $\ell < \tfrac 1 2 \tx{gon}(X_H)-1$, we can apply \Cref{uandTlprop} to deduce that $T_\ell (P)$ and $T_{\ell}u(P)$ have the same ``shape''. 
Since $u$ is defined over a compositum of quadratic fields and $\ell < \tfrac 1 2 \tx{gon}(X_H)-1$, we can apply \Cref{uandTlprop} to deduce that, under the definition of $\sigma$ given in the same proposition, $T_{\ell}u(P) = u^\sigma T_\ell (P)$, hence the same coefficients appear in the divisors $T_{\ell}u(P)$ and $T_{\ell}(P)$, since $u^\sigma$ is an automorphism.
 Hence, applying \Cref{theo:repetitions} once again, we have that $u(P)$ is a cusp if $P$ is a cusp and $u(P)$ is a branch point if $P$ is branch point. 
		
Therefore we proved that $u$ preserves the set of cusps and the set of branch points. Applying \cite[Proposition 3.1]{DoseAut}, we obtain that $u$ is induced by an automorphism $v \colon \HH \to \HH$. We know that $\Aut(\HH)= \tx{PGL}_2^+(\R)$, hence $v$ is a M\"obius transformation given by a matrix $m \in \tx{PGL}_2^+(\R)$. Since it factors through the quotient, $m$ belongs to the normalizer of the image of $\Gamma_H$ in $\tx{PGL}_2^+(\R)$. Hence the restriction of $u$ to $Y_H$ is modular and, by \Cref{PGL2Q-R}, $u$ itself is modular.
\end{proof}

We still need to determine which are the modular automorphisms of a modular curve $X_H$ for Cartan and Cartan-plus subgroups $H$ of $\GLn$. Since in these cases we have $\det(H) =(\Z/n\Z)^\times$, then $Y_H$ also parametrizes pairs $(E,\phi)$ such that the Weil pairing of $(\phi\left(\begin{smallmatrix}1\\0\end{smallmatrix}\right), \phi\left(\begin{smallmatrix}0\\1\end{smallmatrix}\right))$ is fixed, up to the action of $H{\cap}\tx{SL}_2(\Z/n\Z)$. With this interpretation, every matrix $\gamma \in \tx{SL}_2(\Z/n\Z)$ that normalizes $H{\cap}\tx{SL}_2(\Z/n\Z)$ defines an automorphism of $Y_H$ sending $(E,\phi) \mapsto (E, \phi\circ\gamma)$: such an automorphism is modular, induced by a lift of $\gamma$ in $\tx{SL}_2(\Z)$. 
The next proposition implies that these are all the modular automorphisms except when $n \equiv 2 \bmod 4$ and $H$ is a Cartan-plus which is split at $2$. We now suppose we are in this last case and we construct another modular automorphism. Letting $n= 2n'$, we have  
\[
H = H_2 \times H_{n'} \subset \GL(\Z/2\Z) \times \GL(\Z/n' \Z) = \GLn,
\] 
where $H_2$ and $H_{n'}$ are the images of $H$ in $\GL(\Z/2\Z)$ and $\GL(\Z/n'\Z)$ respectively. Since we are assuming that $H_2$ is a split Cartan-plus subgroup, there are three possibilities for $H_2$ (all conjugated) and, depending on them, we define
\begin{equation}\label{eq:ugly_matrix_1}
\gamma_0 := \begin{cases}
\smt 3111, \quad \tx{if }H_2 = \{\tx{Id}, \smt0110 \},\\
\smt 2212, \quad \tx{if }H_2 = \{\tx{Id}, \smt1101 \},\\
\smt 2122, \quad \tx{if }H_2 = \{\tx{Id}, \smt1011 \}.
\end{cases}
\end{equation}
Since the projection $\tx{SL}_2(\Z)\to \tx{SL}_2(\Z/2n\Z) = \tx{SL}_2(\Z/4\Z){\times}\tx{SL}_2(\Z/n'\Z)$ is surjective and since $\det(H_{n'}) = (\Z/n'\Z)^\times$, there exists
\begin{equation}\label{eq:ugly_matrix_2}
\gamma_1 \in \tx{SL}_2(\Z) \quad \tx{such that} \quad \gamma_1 \equiv \smt 1001 \!\!\!\!\!\pmod 4 \quad \tx{and} \quad  (\gamma_1\gamma_0)^T \!\!\!\!\!\pmod{n'}  \in H_{n'}. 
\end{equation}
The matrix $\mathbb{P}(\gamma_1\gamma_0)$ lies in the normalizer $\mathcal N$ of $\PP(\Gamma_H)$ inside $\tx{PGL}_2^+(\Q)$ and we have that $\mathbb{P}(\gamma_1\gamma_0)^2\in\PP(\Gamma_H)$, hence $\mathbb{P}(\gamma_1\gamma_0)$ induces an involution on $X_H$. Since $\PP(\gamma_1\gamma_0) \notin \PP(\tx{SL}_2(\Z))$, the modular automorphism defined by $\gamma_1\gamma_0$ is not of the form $(E,\phi) \mapsto (E, \phi\circ\gamma)$ with $\gamma \in \tx{SL}_2(\Z/n\Z)$.

\begin{prop}\label{prop:Cartan_normalizer_in_GL2Q}
Let $n$ be a positive integer and let $H<\GLn$ be either a Cartan subgroup or a Cartan-plus subgroup. Let $N'<\mathrm{SL}_2(\Z/n\Z)$ be the normalizer of the group $H':=H{\cap}\tx{SL}_2(\Z/n\Z)$ and let $\mathcal N$ be the normalizer of $\PP(\Gamma_H)$ in $\tx{PGL}_2^+(\Q)$.
If $n\equiv 2 \bmod 4$ and $H$ is a Cartan-plus split at $2$, then, for every choice of $\gamma_0$ and $\gamma_1$ as in (\ref{eq:ugly_matrix_1}) and (\ref{eq:ugly_matrix_2}), $\mathcal N$ is generated by $\PP(\Gamma_{N'})$ and $\mathbb{P}(\gamma_1\gamma_0)$. 
Otherwise $\mathcal N$ is $\PP(\Gamma_{N'})$.
\end{prop}
\begin{proof}
Let $\tilde{\mathcal N}< \GL^+(\Q)$ be the normalizer of $\Q^\times \Gamma_H$, or, equivalently, the normalizer of $\Gamma_H$ (each matrix normalizing $\Q^\times \Gamma_H$ also normalizes $(\Q^\times \Gamma_H)\cap \tx{SL}_2(\Q) = \Gamma_H$, and since scalar matrices commute with everything, each matrix normalizing $\Gamma_H$ also normalizes $\Q^\times \Gamma_H$). The conclusion of the proposition is equivalent to 
\[
\tilde{\mathcal N}= \Q^\times \Gamma_{N'} \quad \tx{or}  \quad  \tilde{\mathcal N}= \Q^\times \langle \gamma_1\gamma_0 ,\Gamma_{N'} \rangle ,
\]
depending on the case. The inclusions $\supseteq$ are trivial, hence we prove the other inclusions. 
Since the normalizer of $\Gamma_H$ inside $\tx{SL}_2(\Z)$ is $\Gamma_{N'}$, it is enough to show that 
\[
\tilde{\mathcal N}\subseteq \Q^\times \text{SL}_2(\Z) \quad \tx{or} \quad \tilde{\mathcal N}\subseteq \Q^\times \text{SL}_2(\Z) \cup \gamma_1\gamma_0 \Q^\times \text{SL}_2(\Z),
\] 
depending on the case. We suppose that $\tilde{\mathcal N}$ contains a matrix $m = \smt{a}{b}{c}{d}$ not lying in $\Q^\times \text{SL}_2(\Z)$: it is enough to prove, with this assumption, that $n \equiv 2 \bmod 4$ and $H$ is a Cartan-plus subgroup split at $2$ and $m\in \gamma_1\gamma_0 \Q^\times \text{SL}_2(\Z)$.

Up to multiplication by a scalar matrix, we can suppose that $a,b,c,d\in\Z$ and that $\gcd(a,b,c,d)=1$. Since $m\notin \Q^\times \text{SL}_2(\Z)$, then $\det(m) \neq 1$. Let $p$ be a prime dividing $\det(m)$, let $\lambda_1=\left( \begin{smallmatrix}	a \\ c \end{smallmatrix}\right), \lambda_2=\left( \begin{smallmatrix} b \\ d \end{smallmatrix}\right) \in \Z^2$ and let $\Lambda \subset \Z^2$ be the lattice generated by $\lambda_1, \lambda_2$. By definition of $\tilde{\mathcal N}$, for every $\gamma \in \Gamma_H$ there is $\gamma' = \smt{x}{y}{z}{w}\in\Gamma_H$ such that $\gamma m = m \gamma'$. Hence, looking at the columns of $\gamma m$, we get $\gamma \lambda_1 = x\lambda_1{+}z\lambda_2 $ and $\gamma \lambda_2 =  y\lambda_1{+}w\lambda_2$.
	Since $\gamma$ is arbitrary and $\gamma' \in \tx{SL}_2(\Z)$, we have
	\begin{equation*} \label{eq:Lambda_m_invariant}
	\Gamma_H \Lambda = \Lambda.
	\end{equation*}
Let $\ol \Lambda$ be the image of $\Lambda$ under the quotient map $\Z^2 \to \F_p^2$. Since at least one of $a,b,c,d$ is not multiple of $p$, we know that $\ol \Lambda \neq \{ 0\}$ and since $\det(m)$ is multiple of $p$, we know that $\ol \Lambda \neq \F_p^2$. Hence $\ol \Lambda$ is a line inside $\F_p^2$ which is left invariant by every matrix in the image $\ol{\Gamma}_H$ of $\Gamma_H$ in $\GL(\F_p)$. This implies that $\ol{\Gamma}_H$ is contained in a Borel subgroup of $\GL(\F_p)$, thus $p$ divides the level $n$ and $\ol{\Gamma}_H = \ol H^T {\cap} \tx{SL}_2(\Z/p\Z)$, where $\ol H$ is the image of $H$ in $\GL(\F_p)$. We deduce that either $H$ is a Cartan group split at $p$ or $p=2$ and $H$ is a Cartan-plus group split at $p$.

First we suppose that $H$ is a Cartan group split at $p$. Let $p^e$ be the maximum power of $p$ dividing $n$. Up to conjugacy, the image of $H$ in $\GL(\Z/p^e\Z)$ is $\{\smt{*}{0}{0}{*}\}$, hence $m^{-1}\gamma m \equiv \smt{*}{0}{0}{*} \pmod{p^e}$, as $\gamma\in\Gamma_H$. Applying this to $\gamma = \smt{1}{n}{0}{1}$ and $\gamma = \smt{1}{0}{n}{1}$, we see, since $\det(m)$ is a multiple of $p$, that $p\mid a,b,c,d$, which is a contradiction.

This contradiction implies that the only prime dividing $\det(m)$ is $2$ and $H$ is a Cartan-plus group split at $2$. Let $2^e$ be the maximum power of $2$ dividing $n$. Up to conjugacy, the image of $H$ in $\GL(\Z/2^e\Z)$ is $\{\smt{*}{0}{0}{*},\smt 0**0 \}$. In particular the image of $H$ in $\GL(\Z/2\Z)$ is $\{\smt{1}{0}{0}{1}, \smt{0}{1}{1}{0}\}$, hence $\ol{\Lambda} = \langle\left(\begin{smallmatrix}1\\1\end{smallmatrix}\right)\rangle$ is the only $\ol{\Gamma}_H$-invariant line. With a similar argument we see that the rows $(a\,b), (c\,d)$ of $m$ span $\langle(1\,1)\rangle$ in $\F_2^2$. Hence $m \equiv \smt 1111 \pmod 2$. For every $\gamma\in\Gamma_H$, we have
\begin{equation}\label{eq:}
m^{-1}\gamma m \!\!\!\! \pmod{2^e} \in \left\{\smt{*}{0}{0}{*},\smt 0**0\right\}.
\end{equation}
When $\gamma = \smt{1}{n}{0}{1}$, we see that $m^{-1}\gamma m \equiv \smt*00* \pmod{2^e}$ is not possible because both $c$ and $d$ are odd, hence $m^{-1}\gamma m \equiv \smt 0**0 \pmod{2^e}$ and, by explicit computations, we deduce that $\det(m)=2$ and $n \equiv 2 \bmod 4$. Finally, since $m \equiv \smt 1111 \pmod 2$ and $\det(m)=2$, we see that $(\gamma_1\gamma_0)^{-1}m\in \tx{SL}_2(\Z)$.
\end{proof}

We now prove the main results of this paper.

\begin{thm}\label{thm:Cartan_aut}
Let $n\ge 10^{400}$ be an integer and let $H<\GLn$ be either a Cartan or a Cartan-plus subgroup. Then every automorphism of $X_H$ is modular, hence we have
\begin{align*}
\Aut(X_H) &\cong \begin{cases}  N'/H' \times \Z/2\Z, & \tx{if } n\equiv 2 \bmod 4 \tx{ and $H$ is a Cartan-plus split at }2, \\
N'/H', & \tx{otherwise,}
\end{cases}
\end{align*}
where $N'<\mathrm{SL}_2(\Z/n\Z)$ is the normalizer of $H':=H\cap\tx{SL}_2(\Z/n\Z)$.
\end{thm}
\begin{proof}

Let $\mathcal N$ be the normalizer of $\PP(\Gamma_H)$ inside $\mathrm{PGL}_2^+(\Q)$. By Proposition \ref{prop:Cartan_normalizer_in_GL2Q}
we have
\[
\mathcal N/\PP(\Gamma_H) {\cong} \begin{cases} \PP(\Gamma_{N'})/\PP(\Gamma_H) {\times} \Z/2\Z, & \tx{if } n\equiv 2 \bmod 4 \tx{ and $H$ is a Cartan-plus split at }2, \\
\PP(\Gamma_{N'})/\PP(\Gamma_H), & \tx{otherwise},
\end{cases}
\]
where the first case is true because  $\mathbb{P}(\gamma_1\gamma_0 \Gamma_H)$ has order $2$ in $\mathcal N/\PP(\Gamma_H)$ and commutes with every element in $\PP(\Gamma_{N'})/\PP(\Gamma_H)$. Since $\PP(\Gamma_{N'})/\PP(\Gamma_H) \cong \PP(\Gamma_{N'})/\PP(\Gamma_{H'}) \cong N'/H'$, it is enough to prove that every automorphism of $X_H$ is modular. For $n\ge 10^{400}$ every automorphism is defined over the compositum of some quadratic fields by Proposition \ref{prop:field_Cartan(n)}. We can bound the gonality $\tx{gon}(X_H)$ of $X_H$ using \cite{AbrGonality} and, with the same estimates used in the proof of Proposition \ref{prop:genus_bounds}, we have
	\[ 
	\tx{gon}(X_H) \geq \frac{7}{800}[\tx{SL}_2(\Z):\Gamma_H] \geq \frac{7 n^2}{800 (\omega(n){+}1)2^{\omega(n)}} > 10 n.
	\]
	So, there is at least one prime $\ell$ not dividing $n$ with $5\leq \ell < \tfrac{1}{2}\tx{gon}(X_H){-}1$. By Lemma~\ref{lem:aut_general_H}, we can conclude that every automorphism is modular.
\end{proof}

\begin{rem}\label{rem:N'}
One can determine the groups $N'/H'$ in all cases. Indeed, let $n=\prod_{i=1}^r p_i^{e_i}$ be any positive integer with its prime factorization, let $H < \GLn$ be either a Cartan or a Cartan-plus subgroup and let $N'<\mathrm{SL}_2(\Z/n\Z)$ be the normalizer of the group $H':= H \cap \mathrm{SL}_2(\Z/n\Z)$. By Chinese Remainder Theorem we have 
\[
H' \cong \prod_{i=1}^r H'_i \quad \tx{and} \quad N' \cong \prod_{i=1}^r N'_i \quad \tx{inside}  \quad \mathrm{SL}_2(\Z/n\Z)\cong \prod_{i=1}^r \mathrm{SL}_2(\Z/p^{e_i}\Z),
\]
where $H_i'$ is the image of $H'$ in $\mathrm{SL}_2(\Z/p^{e_i}\Z) $ and $N'_i < \mathrm{SL}_2(\Z/p^{e_i}\Z)$ is the normalizer of~$H_i'$. Hence the knowledge of $N'/H'$ for $H \in \{C_\ns(p^e), C_\ns(p^e), C_\tx{s}(p^e), C_\tx{s}^+(p^e)\}$ allows to compute the group $N'/H'$ for every Cartan or Cartan-plus subgroup $H$ of level $n$ not necessarily a prime power. For the prime power cases see \Cref{lem:N'} and Table \ref{table2} below.
\end{rem}
\begin{lem}\label{lem:N'}
Let $e$ be a positive integer and let $p$ be a prime. Let $H < \mathrm{GL}_2(\Z/p^e\Z)$ be either a Cartan or a Cartan-plus subgroup and let $N'<\mathrm{SL}_2(\Z/p^e\Z)$ be the normalizer of the group $H':= H \cap \mathrm{SL}_2(\Z/p^e\Z)$. Then:
\begin{itemize}
\item if $H = C_\ns(p^e)$, then $N'/H'\cong \Z/2\Z$, since $N'= C_\ns^+(p^e)\cap \SLpe$;
\item if $H = C_\ns^+(p^e)$ and $p^e \neq 3$, then $N'/H'\cong \{1\}$;
\item if $H = C_\tx{s}(p^e)$ and $p\neq 2,3$, then $N'/H'\cong\langle \smt 0{-1}10\rangle\cong \Z/2\Z$;
\item if $H=C_\tx{s}^+(p^e)$ and $p\neq 2,3$ and $p^e \neq 5$, then $N'/H'\cong \{1\}$.
\end{itemize}
The cases left out are listed in Table \ref{table2} below.
\end{lem}
\begin{proof}
It is a direct computation.
\end{proof}
\begin{table}[h!]  
\renewcommand*{\arraystretch}{1.2}
\caption{Automorphism groups.}
\begin{center}
\begin{tabular}{c|c|c|c}
	\toprule
	$H$  & $N'/H'$ & Generators & Comments
	\\ \midrule\midrule 
	$C_\ns^+(3)$  & $\Z/3\Z$ & $\smt1101$  & \\
	\hline 
	$C_\tx{s}(2^e)$ & \multirow{2}{*}{$\mathrm{PSL}_2(\Z/2^e\Z)$} & \multirow{2}{*}{$\mathrm{SL}_2(\Z/2^e\Z)$} & \\
	$1{\le} e{\le} 3$ &  &  &  \\
	\hline 
	\multirow{4}{*}{$C_\tx{s}(2^4)$}  & $D_8 \rtimes_\varphi (\Z/8\Z)$  &  \multirow{2}{*}{$\smt{-1}{6}{6}{-5}\!,\!\smt{4}{9}{7}{-4}$ for $D_8$}  & $D_8\cong \Z/8\Z\rtimes\Z/2\Z$ is the  \\
	 & with  &   &  dihedral group of order $16$; \\
	 & $\left(\varphi(1)\right)(1,0)=(5,0)$ & \multirow{2}{*}{and $\smt {1}{-2}{0}{1}$ for $\Z/8\Z$} &  $N'/H'$ is the group labeled as  \\
	 & $\left(\varphi(1)\right)(0,1)=(3,1)$ &  &   $(128,68)$ in MAGMA, \cite{GroupNames}  \\
	\hline 
	$C_\tx{s}(2^e)$ & $(\Z/8\Z)^2 \rtimes_\varphi (\Z/2\Z)$ & $\smt{1}{2^{e-3}}{0}{1}\!,\!\smt{1}{0}{-2^{e-3}}{1}$ for $(\Z/8\Z)^2 $ & this group is labeled as \\
	 $e\ge 5$  & with $\left(\varphi(1)\right)(x,y){=}(y,x)$ & and $\smt {0}{-1}{1}{0}$ for $\Z/2\Z$ &  $(128,67)$ in MAGMA, \cite{GroupNames} \\
	\hline 
	$C_\tx{s}(3)$ & $\mathrm{PSL}_2(\Z/3\Z)$ & $\mathrm{SL}_2(\Z/3\Z)$  &\\
	\hline 
	$C_\tx{s}(3^e)$  & \multirow{2}{*}{$\Z/3\Z \times S_3$} & $\smt{1}{3^{e-1}}{-3^{e-1}}{1}$ for $\Z/3\Z$ & $S_3$ is the symmetric group  \\
	$e\ge 2$  & & and $\smt {0}{-1}{1}{0}\!,\!\smt{1}{3^{e-1}}{3^{e-1}}{1}$ for $S_3$ & acting on three elements \\
	\hline 
	$C_\tx{s}^+(2)$  & $\{1\}$ &   & \\
	\hline 
	$C_\tx{s}^+(2^2)$  & $\Z/2\Z$ & $\smt{1}{2}{2}{1}$ & \\
	\hline 
	$C_\tx{s}^+(2^3)$  & $\Z/4\Z$ & $\smt{1}{-2}{2}{-3}$ & \\
	\hline 
	$C_\tx{s}^+(2^4)$  & $\Z/8\Z$ & $\smt{1}{6}{2}{-3}$ & \\
	\hline 
	$C_\tx{s}^+(2^5)$  & $\Z/8\Z$ & $\smt{1}{-4}{4}{-15}$ & \\ 
	\hline 
	$C_\tx{s}^+(2^e)$  & \multirow{2}{*}{$\Z/8\Z$} & \multirow{2}{*}{$\smt {1}{-2^{e-3}}{2^{e-3}}{1}$} & \\ 
	$e\ge 6$  &  & & \\ 
	\hline 
	$C_\tx{s}^+(3)$  & $\Z/2\Z$ & $\smt {1}{1}{1}{-1}$ & \\ 
	\hline 
	$C_\tx{s}^+(3^e)$ & \multirow{2}{*}{$\Z/3\Z$} & \multirow{2}{*}{$\smt {1}{-3^{e-1}}{3^{e-1}}{1}$} & \\ 
	$e\ge 2$  &  &  & \\ 
	\hline 
	$C_\tx{s}^+(5)$  & $\Z/3\Z$ & $\smt 1213$ & \\ 
	\bottomrule
\end{tabular}
\end{center}
\label{table2}
\end{table}
Note that the groups $N'/H'$ computed for $H=C_\tx{s}(p^e)$ are the same determined in \cite{AL}, \cite{ASNormGamma0}, \cite{BarsNormGamma0}, in the setting of Borel modular curves.

For Cartan modular curves of prime power level we make \Cref{thm:Cartan_aut} more precise.

\begin{thm}\label{thm:powerprimes}
Let $p$ be a prime number and let $e$ be a positive integer. If $p^e>11$ and $p^e\notin\{3^3, 2^4, 2^5, 2^6\}$, then all the automorphisms of $X_\ns(p^e), X_\ns^+(p^e), X_\tx{s}(p^e)$ and $X_\tx{s}^+(p^e)$ are modular and
\[
\begin{aligned}
 \Aut(X_\ns(p^e)) &\cong \Z/2\Z, &
 \Aut(X_\ns^+(p^e)) &\cong \{1 \}, \\
 \Aut(X_\tx{s}(p^e)) &\cong \begin{cases}
(\Z/8\Z)^2 \rtimes (\Z/2\Z), & \tx{ if }p=2, \\
\Z/3\Z\times S_3, & \tx{ if }p=3,  \\
\Z/2\Z, & \tx{ if }p>3, \end{cases} &
 \Aut(X_\tx{s}^+(p^e)) &\cong \begin{cases}
\Z/8\Z, & \tx{ if }p=2, \\
\Z/3\Z, & \tx{ if }p=3,  \\
\{1\}, & \tx{ if }p>3,
\end{cases}
\end{aligned}
\]
where the above semidirect product $(\Z/8\Z)^2 \rtimes \Z/2\Z$ is described in Table \ref{table2}.
\end{thm}
\begin{proof}
We first treat the case $p^e> 49$ with $p^e\ne 2^6=64$. Up to conjugacy we can assume that $H \in \{C_\tx{s}(p^e), C_\tx{s}^+(p^e) , C_\ns(p^e), C_\ns^+(p^e) \}$ where these groups are the subgroups of $\GL(\Z/p^e\Z)$ defined in Chapter \ref{Sec:Cartan} and $X_H \in \{X_\ns(p^e), X_\ns^+(p^e), X_\tx{s}(p^e), X_\tx{s}^+(p^e) \}$ is the corresponding associated modular curve. By \cite[Theorem 0.1]{AbrGonality} and Table \ref{table1}, for $p^e>87$,  we have the following lower bounds for the gonality of $X_H$:
\[
\tx{gon}(X_H) \ge  \frac{7}{800}[\tx{SL}_2(\Z):\Gamma_H] \ge \frac{7}{800} \frac{p^{2e}(1-\tfrac1p)}{2}>\frac{7\cdot 87^2  }{3200} > 16.
\]
Hence either the prime $\ell=5$ or the prime $\ell=7$, are different from $p$, and satisfy $5\le \ell<\tfrac{1}{2}\tx{gon}(X_H)-1$. With a similar computation one can show that $\tx{gon}(X_H)>12$, for $49<p^e\le 87$, if $p^e\ne 64$ and we can take $\ell=5$. Applying Lemma \ref{lem:aut_general_H} we deduce that all the automorphisms of $X_H$ defined over a compositum of quadratic fields are modular, hence, by Proposition \ref{prop:field_Cartan(p^e)}, all the automorphisms of $X_H$ are modular. Finally, we can use Proposition \ref{prop:Cartan_normalizer_in_GL2Q},  \Cref{lem:N'} and \Cref{rem:N'} to obtain the group of modular automorphisms.

We now assume $11<p^e \le 49$. All the cases $X_\tx{s}(p^e) \cong X_0(p^{2e})$ are treated in \cite{KM}, all the cases $X_\tx{s}^+(p)$ are treated in \cite{GonzSplit} and the cases $X_\ns(p)$, $X_\ns^+(p)$, for $13 \le p \le 31$, are treated in \cite{GonzConst}. The remaining cases  $X_\tx{s}^+(25)$, $X_\tx{s}^+(49)$ and $X_\ns(p^e)$, $X_{\ns}^+(p^e)$, for $p^e = 25,37,41,43,47,49$, are treated in the MAGMA script available at \cite{Script_aut}.
\end{proof}

Last theorem can be specialized to the prime level case, obtaining new results for non-split Cartan curves. The split cases are treated in \cite{GonzSplit} and \cite{KM}.

\begin{cor}\label{cor:aut_prime_level}
Let $p \ge 13$ be a prime number. Then the group of automorphisms of $X_\ns^+(p)$ is trivial and the group of automorphisms of $X_\ns(p)$ has order $2$.
\end{cor}

\begin{rem}
Theorem \ref{thm:powerprimes} implies that, for $p^{2e}$ big enough, all the automorphisms of $X_0^*(p^{2e})\cong X_\tx{s}^+(p^e)$ are modular, extending \cite{BakerHas} and \cite{GonzSplit} that treat the cases $X_0^*(p)$ and $X_0^*(p^2)$ and complementing in part \cite{BarsGonzAut} which treats the case of $X_0^*(n)$ for $n$ squarefree. Our techniques (in particular Lemma \ref{lem:aut_general_H})  cannot be generalized to the case $X_0^*(p^e)$ with $e$ odd, because some branch points of the natural map $\HH\to Y_0^+(p^e)$ have the form $\{(E,C),(E/C,E[p^e]/C)\}$ with $E\neq E_i,E_\rho$. 
Anyway, the techniques used in \cite[Lemmas 4, 5, 6]{GonzSplit}, together with Proposition \ref{prop:field_X0+(p^e)}, can be used to prove the modularity of all elements in $\Aut(X_0^*(p^e))$, without restrictions on $e$, for all but finitely many cases.
\end{rem}
It is a natural question to ask whether modular automorphisms such as those described in \Cref{lem:N'} are defined on a small field. The next proposition partially addresses this issue.
\begin{prop} \label{rem:field_supermodular}
Let $n$ be a positive integer, let $H$ be a subgroup of $\GLn$ such that $\det(H)= (\Z/n\Z)^\times$ and let  $H'=H \cap \text{SL}_2(\Z/n\Z)$. Let $M\in \text{SL}_2(\Z)$ such that its reduction $M\bmod n$ normalizes $(H')^T$. Then $M$ defines a modular automorphism of $X_H$ which is defined over the cyclotomic field $\Q(\zeta_n)$. Moreover, this automorphism is defined over $\Q$ if and only if $M\bmod n$ normalizes $H^T$.
\end{prop}
\begin{proof}
Since $M$ normalizes $\Gamma_H$, it defines a modular automorphism $\Psi$ of $X_H$. From now on, we only look at the restriction of $\Psi$ to $Y_H$. Let $Y_{H'}^{cc}$ be the connected component of $Y_{H'}$ such that $Y_{H'}^{cc}=\Gamma_H\backslash \HH$. Hence the natural map $Y_{H'}\to Y_H$ restricts to an isomorphism between $Y_{H'}^{cc}$ and $Y_H$ defined over $\Q(\zeta_n)$, which is also the field of definition of $Y_{H'}^{cc}$. The curve $Y_{H'}^{cc}/\Q(\zeta_n)$ can be interpreted as the coarse moduli space of elliptic curves with $H'$-structure $(E,\phi)$ such that the Weil pairing of $(\phi\left(\begin{smallmatrix}1\\0\end{smallmatrix}\right), \phi\left(\begin{smallmatrix}0\\1\end{smallmatrix}\right))$ is a fixed root of unity $\zeta$. Moreover, under the isomorphism $Y_H \cong Y_{H'}^{cc}$, we can describe $\Psi$ as  
	\begin{equation}\label{eq:super_mod} 
	\Psi \colon Y_{H'}^{cc} \to Y_{H'}^{cc}, \quad  (E,\phi)\mapsto (E, \phi\circ m),
	\end{equation}
	with $m:=M^{-T} \bmod n$. Since the map $(E,\phi)\mapsto (E, \phi\circ m)$ also defines a natural transformation between the functor coarsely represented by $Y_{H'}^{cc}$ and itself, we deduce that $\Psi$ is defined over $\Q(\zeta_n)$.
	
Notice that for every point $P\in Y_H(\C)$ the image $\Psi(P)$ can be described by \Cref{eq:super_mod} a priori only if we choose for $P$ a representative $(E,\phi)$ such that the Weil pairing of $(\phi\left(\begin{smallmatrix}1\\0\end{smallmatrix}\right), \phi\left(\begin{smallmatrix}0\\1\end{smallmatrix}\right))$ is $\zeta$. 

We now prove that the map $\Psi$ is defined over $\Q$ if and only if $m$ normalizes the whole $H$. One implication is trivial, since, when  $m$ normalizes $H$, the map $(E,\phi)\mapsto (E, \phi\circ m)$ also defines a natural transformation between the functor coarsely represented by $Y_H$ and itself. For the other implication, suppose that $m$ is in the normalizer of $H'$ but not in the normalizer of $H$ and let $h \in H$ be such that $m^{-1}hm \notin H$. Let $(E,\phi)$ be a point in $Y_{H'}^{cc}(\ol \Q)$ such that $E$ is defined over $\Q$. Then, in $Y_H(\ol \Q)$, for every $\sigma \in \Gal(\ol \Q/\Q)$, we have 
\[ \sigma (E, \phi) =  (E, \sigma \circ \phi) =  (E, \sigma \circ \phi \circ h ).
\]
If we choose $\sigma$ such that $\sigma(\zeta_n) = \zeta_n^{\det h^{-1}}$, then the Weil pairing of $(\sigma {\circ} \phi {\circ} h \left(\begin{smallmatrix}1\\0\end{smallmatrix}\right),      \sigma {\circ} \phi {\circ} h  \left(\begin{smallmatrix}1\\0\end{smallmatrix}\right))$ is the same as the Weil pairing of $(\phi\left(\begin{smallmatrix}1\\0\end{smallmatrix}\right), \phi\left(\begin{smallmatrix}0\\1\end{smallmatrix}\right))$ which is equal to $\zeta$. Hence, we can use \Cref{eq:super_mod} to deduce that
\[
\begin{aligned}
\sigma (\Psi (E,\phi)) &= \sigma(E, \phi\circ m) = (E, \sigma \circ \phi\circ m), \\
\Psi (\sigma (E,\phi)) &= \Psi (E, \sigma \circ \phi \circ h) = (E, \sigma \circ \phi \circ h\circ m).
\end{aligned}
\]
If $\Psi$ was defined over $\Q$, we would have $\sigma (\Psi (E,\phi))=\Psi (\sigma (E,\phi))$, implying that $m^{-1} h m$ belongs to $ H$ which is a contradiction.
\end{proof}
\begin{rem}
The proposition above, together with \Cref{thm:powerprimes}, implies that, given a  prime power $p^e>11$ not in $\{3^3, 2^4, 2^5, 2^6\}$, all the automorphisms of a Cartan curve of level $p^e$ are defined over the cyclotomic field $\Q(\zeta_{p^e})$. In general, not all such automorphisms are defined over $\Q$, even though this is true when the prime $p$ is at least 5 or when the Cartan group is non-split.
\end{rem}

\newpage

\section{Appendix}\label{sec:appendix}

\subsection*{Characters of $\GL(\Z/2^e\Z)$}
Let $G:=\GL(\Z/2^e\Z)$. For each $H<G$, let $\chi_H \colon G \to \Q$ be the character of the representation $\Q[G/H] = \bigoplus gH \cdot \Q$, computed using \Cref{eq:characters}. 
Every element of $G$ is conjugated to a unique element appearing in the first column, hence the table determines the characters $\chi_H$ for $H$ appearing in \Cref{lem:reprisom} or in \cite[Theorem~$1.1$]{ChenPowers}. In the first column we have $\lambda, a \in (\Z/2^e\Z)^\times$, $b \in (\Z/2^e\Z)$, $k\in \{ 1,\ldots ,e{-}1\}$, and $u \in (\Z/2^{e-k}\Z)^\times$.
%\newpage
\setlength\tabcolsep{2 pt}
\renewcommand*{\arraystretch}{1.3}
\begin{center}
\captionof*{table}{Character table.}
\begin{tabular}{c||c|c|c|c|c|c|c } 
	\toprule
	& $B_r, r \ge 0$ & $T_0$ & $T_r, r > 0$ & $C_s$ & $C_s^+$ & $C_\ns$ & $C_\ns^+$ \\
	\midrule
	\midrule
	$\lambda \tx{Id}$
	& $3{\cdot}2^{2r}$ & 1 & $3{\cdot}2^{2r-1}$ & $ 3{\cdot}2^{2e-1} $ & $3{\cdot}2^{2e-2}$ & $2^{2e-1}$ & $2^{2e-2}$ \\ 
	\hline
	$\!\begin{array}{l} \smt 0a1b \\ b \tx{ odd} \end{array}$  %CASE
	& %B_r
	0
	&1& %T_r
	0
	& %C_s
	0 
	& %C_s^+
	0
	& %C_ns
	2
	& %C_ns^+
	1
	\\
	\hline
	$\!\begin{array}{l} \smt 0a1b  \\ b \tx{ even} \end{array}$
	& %B_r
	$\!\begin{array}{cc} 
	1 \!& \!\tx{if }r {=} 0 \\
	0 \!& \!\tx{if }r {>} 0
	\end{array}$
	&1& %T_r
	0
	& %C_s
	0
	& %C_s^+
	$\!\begin{array}{cc}
	2^{e-1}  \!& \!\tx{if }b {=} 0 \\
	0  \!& \!\tx{if }b {\neq} 0
	\end{array}$
	& %C_ns
	0
	& %C_ns^+
	$\!\begin{array}{cc} 
	2^{e-1} \!& \!\tx{if }b {=} 0 \\
	0 \!& \!\tx{if }b {\neq} 0
	\end{array}$ \\
	\hline
	$\smt{\lambda}{0}{0}{\lambda+2^k u}$ 
	& %B_r
	$ \!\begin{array}{cc} 3{\cdot}2^{2r} \!& \!\tx{if }r{<}k \\   2^{2k+1}  \!& \!\tx{if }r{\ge}k \end{array}$
	&1& %T_r
	$ \!\begin{array}{cc} 3{\cdot}2^{2r-1} \!& \!\tx{if }r{\le}k \\ 2^{2k+1} \!& \!\tx{if }r{>}k\end{array}$
	& %C_s 
	$2^{2k+1}$
	& %C_s^+
	$2^{2k}$ 
	& %C_ns
	0
	& %C_ns^+
	0
	\\
	\hline
	$\smt{\lambda}{2^k u}{2^k}{\lambda}$
	& %B_r
	$\!\begin{array}{cc} 3{\cdot}2^{2r} \!& \!\tx{if }r{<}k \\  2^{2r} \!& \!\tx{if }r{=}k \\ 0  \!& \!\tx{if }r{>}k \end{array}$
	&1& %T_r
	$\!\begin{array}{cc} 3{\cdot}2^{2r-1} \!& \!\tx{if }r{\le}k \\ 0 \!& \!\tx{if }r{>}k \end{array}$
	& %C_s
	0
	& %C_s^+
	0
	& %C_ns
	0
	& %C_ns^+
	0
	\\
	\hline
	$\smt{\lambda}{2^k u}{2^k}{\lambda+2^k}$
	& %B_r
	$\!\begin{array}{cc} 3{\cdot}2^{2r} \!& \!\tx{if }r{<}k \\ 0 \!& \!\tx{if }r{\ge}k \end{array}$
	&1& %T_r
	$\!\begin{array}{cc} 3{\cdot}2^{2r-1} \!& \!\tx{if }r{\le}k \\ 0 \!& \!\tx{if }r{>}k \end{array}$
	& %C_s
	0
	& %C_s^+
	0
	& %C_ns
	$2^{2k+1}$
	& %C_ns^+
	$ 2^{2k}$
	\\
	\hline
	$\smt{\lambda}{2^k u}{2^k}{\lambda+2^{k+1}}$
	& %B_r
	$ \!\begin{array}{cc} 3{\cdot}2^{2r} \!& \!\tx{if } r{<}k \\ 2^{2r} \!& \!\tx{if }r{=}k \\  0 \!& \!\tx{if }r{>}k \end{array}$
	&1& %T_r 
	$ \!\begin{array}{cc} 3{\cdot}2^{2r-1} \!& \!\tx{if }r {\le}k \\ 0  \!& \!\tx{if }r{>}k \end{array}$
	& %C_s
	0
	& %C_s^+
	0
	& %C_ns
	0
	& %C_ns^+
	0
	\\ 
	\bottomrule
\end{tabular}
\end{center}
\label{tableApp}

\subsection*{Data for Cartan modular curves of level $n\le64$}
In \Cref{tab:lowlevels}, we collect some relevant data about totally split or totally non-split Cartan modular curves of low level. Here $n$ is the level of the curve, $g$ is the genus, $\widetilde{\mathrm{cm}}$ is the dimension of the CM part after removing all factors isogenous to elliptic curves over $\Q$, and $A$ is the Abramovich's lower bound for the gonality of the curve (see \cite{AbrGonality}). The values were computed through the script at \cite{Script_aut} using the data on modular forms at \cite{lmfdb}.

We write in \textit{italic} the cases of genus $0$ or $1$: in these cases we know that there are infinitely many automorphisms which are not modular. We write in \textbf{bold} the cases where we have $\widetilde{\mathrm{cm}}<\frac{g-1}2$ and $A>2(\ell{+}1)$ for some prime $\ell\ge 5$ not dividing $n$: in these cases we are able to prove that all the automorphisms are modular using \Cref{JacFactorEnd} (we can apply it because of the first inequality) and \Cref{lem:aut_general_H} (we can apply it because of the second inequality). In the remaining cases we are not able to determine whether all the automorphisms are modular just by looking at $g$, $\widetilde{\mathrm{cm}}$ and $A$. We notice that in \Cref{thm:Cartan_aut} we mostly use the hypothesis ``$n\ge10^{400}$'' in order to conclude that the inequality $\widetilde{\mathrm{cm}}<\frac{g-1}2$ holds. On the other hand, one can check that the inequality $A>2(\ell{+}1)$ holds starting from a much smaller $n$. When looking at explicit examples, we notice that,  except for $X_\ns^+(48)$, in all the cases of genus at least $2$ in the table, we have $\widetilde{\mathrm{cm}}<\frac{g-1}2$ and this implies that all the automorphisms are defined over a compositum of quadratic fields.

\begin{rem}\label{rem:ns+16}
We highlight two interesting examples with exceptional automorphisms. The modular curves $\Xns^+(16)$ and $\Xns^+(20)$ have genus two, and hence they are hyperelliptic. An equation for both curves and the description of all rational points can be found in \cite[Section 5]{BaranClass}. We computed the automorphism group of both curves using the built-in function in MAGMA: the automorphism group of $\Xns^+(16)$ is $\Z/2\Z\times\Z/2\Z$ while the automorphism group of $\Xns^+(20)$ is $\Z/2\Z$. All the non-trivial automorphisms are exceptional.
Furthermore, in the case of $\Xns^+(16)$ the exceptional automorphisms give a non-CM rational point.
\end{rem}

\begin{table}
\captionof{table}{Data for low level cases.}\label{tab:lowlevels}
    \begin{minipage}{0.48\textwidth}

\setlength\tabcolsep{0.5 pt}
\renewcommand*{\arraystretch}{1.1}
\begin{center}
%\captionof*{table}{Low level cases: levels 1-32.}
\begin{tabular}{C{0.6cm}|C{0.5cm}C{0.5cm}C{0.4cm}|C{0.5cm}C{0.5cm}C{0.4cm}|C{0.5cm}C{0.5cm}C{0.4cm}|C{0.7cm}C{0.6cm}C{0.5cm} } 
	\toprule
& \multicolumn{3}{c}{$X_{\text{ns}}^+(n)$} & \multicolumn{3}{|c}{$X_{\text{ns}}(n)$} & \multicolumn{3}{|c}{$X_{\text{s}}^+(n)$} & \multicolumn{3}{|c}{$X_{\text{s}}(n)$} \\
	\midrule
	 $n$ & $g$ & $\widetilde{\mathrm{cm}}$ & $A$ & $g$ & $\widetilde{\mathrm{cm}}$ & $A$ & $g$ & $\widetilde{\mathrm{cm}}$ & $A$ & $g$ & $\widetilde{\mathrm{cm}}$ & $A$ \\
	\midrule
	1 & \textit{0}&\textit{0}&\textit{1} & \textit{0}&\textit{0}&\textit{1} & \textit{0}&\textit{0}&\textit{1} & \textit{0}&\textit{0}&\textit{1}\\ 
	\hline
	2 & \textit{0}&\textit{0}&\textit{1} & \textit{0}&\textit{0}&\textit{1} & \textit{0}&\textit{0}&\textit{1} & \textit{0}&\textit{0}&\textit{1}\\ 
	\hline
	3 & \textit{0}&\textit{0}&\textit{1} & \textit{0}&\textit{0}&\textit{1} & \textit{0}&\textit{0}&\textit{1} & \textit{0}&\textit{0}&\textit{1}\\ 
	\hline
	4 & \textit{0}&\textit{0}&\textit{1} & \textit{0}&\textit{0}&\textit{1} & \textit{0}&\textit{0}&\textit{1} & \textit{0}&\textit{0}&\textit{1}\\ 
	\hline
	5 & \textit{0}&\textit{0}&\textit{1} & \textit{0}&\textit{0}&\textit{1} & \textit{0}&\textit{0}&\textit{1} & \textit{0}&\textit{0}&\textit{1}\\ 
	\hline
	6 & \textit{0}&\textit{0}&\textit{1} & \textit{1}&\textit{0}&\textit{1} & \textit{0}&\textit{0}&\textit{1} & \textit{1}&\textit{0}&\textit{1}\\ 
	\hline
	7 & \textit{0}&\textit{0}&\textit{1} & \textit{1}&\textit{0}&\textit{1} & \textit{0}&\textit{0}&\textit{1} & \textit{1}&\textit{0}&\textit{1}\\ 
	\hline
	8 & \textit{0}&\textit{0}&\textit{1} & \textit{1}&\textit{0}&\textit{1} & \textit{1}&\textit{0}&\textit{1} & 3&0&1\\ 
	\hline
	9 & \textit{0}&\textit{0}&\textit{1} & 2&0&1 & \textit{1}&\textit{0}&\textit{1} & 4&0&1\\ 
	\hline
	10 & \textit{0}&\textit{0}&\textit{1} & \textit{1}&\textit{0}&\textit{1} & \textit{1}&\textit{0}&\textit{1} & 7&0&2\\ 
	\hline
	11 & \textit{1}&\textit{0}&\textit{1} & 4&0&1 & 2&0&1 & 6&0&2\\ 
	\hline
	12 & \textit{0}&\textit{0}&\textit{1} & 3&0&1 & 3&0&2 & 13&0&3\\ 
	\hline
	13 & 3&0&1 & 8&0&2 & 3&0&1 & 8&0&2\\ 
	\hline
	14 & \textit{0}&\textit{0}&\textit{1} & 5&0&1 & 3&0&2 & 17&0&3\\ 
	\hline
	15 & \textit{1}&\textit{0}&\textit{1} & 7&2&2 & 4&0&2 & 19&2&4\\ 
	\hline
	16 & 2&0&1 & 7&2&2 & 9&0&2 & 21&2&4\\ 
	\hline
	17 & 6&0&2 & 15&0&3 & 7&0&2 & 17&0&3\\ 
	\hline
	18 & \textit{0}&\textit{0}&\textit{1} & 7&0&1 & 7&0&3 & 37&0&6\\ 
	\hline
	19 & 8&0&2 & 20&0&3 & 9&0&2 & 22&0&4\\ 
	\hline
	20 & 2&0&1 & 9&0&2 & 10&0&4 & 43&0&7\\ 
	\hline
	21 & \textit{1}&\textit{0}&\textit{2} & 15&2&3 & 9&0&3 & 41&2&6\\ 
	\hline
	22 & \textit{1}&\textit{0}&\textit{1} & 13&2&2 & 10&0&4 & 49&2&7\\ 
	\hline
	23 & 13&0&3 & 31&3&5 & 15&0&3 & 35&3&5\\ 
	\hline
	24 & \textit{1}&\textit{0}&\textit{1} & 13&0&2 & 17&0&6 & 73&0&11\\ 
	\hline
	25 & 14&0&3 & 32&0&5 & 22&0&4 & 48&0&7\\ 
	\hline
	26 & 3&0&2 & 21&0&3 & 15&0&5 & 71&0&10\\ 
	\hline
	27 & 12&0&3 & 32&0&5 & 28&0&5 & 64&0&9\\ 
	\hline
	28 & 4&0&2 & 23&0&3 & 21&0&6 & 89&0&12\\ 
	\hline
	29 & 24&0&4 & 54&0&8 & 26&0&4 & 58&0&8\\ 
	\hline
	30 & \textit{1}&\textit{0}&\textit{2} & 17&2&3 & 16&0&10 & \textbf{145}&\textbf{6}&\textbf{19}\\ 
	\hline
	31 & 28&0&5 & 63&3&9 & 30&0&5 & 67&3&9\\ 
	\hline
	32 & 14&2&3 & 35&3&5 & 49&8&7 & \textbf{105}&\textbf{18}&\textbf{14}\\ 
	\bottomrule
\end{tabular}
\end{center}
\label{tablelowlvls1}
    \end{minipage}
    \hfill
    \begin{minipage}{0.49\linewidth}
\setlength\tabcolsep{0.5 pt}
\renewcommand*{\arraystretch}{1.1}
\begin{center}
%\captionof*{table}{Low level cases: levels 33-64.}
\begin{tabular}{C{0.6cm}|C{0.6cm}C{0.6cm}C{0.5cm}|C{0.7cm}C{0.6cm}C{0.5cm}|C{0.7cm}C{0.6cm}C{0.5cm}|C{0.7cm}C{0.6cm}C{0.5cm} }
	\toprule
& \multicolumn{3}{c}{$X_{\text{ns}}^+(n)$} & \multicolumn{3}{|c}{$X_{\text{ns}}(n)$} & \multicolumn{3}{|c}{$X_{\text{s}}^+(n)$} & \multicolumn{3}{|c}{$X_{\text{s}}(n)$} \\
	\midrule
	 $n$ & $g$ & $\widetilde{\mathrm{cm}}$ & $A$ & $g$ & $\widetilde{\mathrm{cm}}$ & $A$ & $g$ & $\widetilde{\mathrm{cm}}$ & $A$ & $g$ & $\widetilde{\mathrm{cm}}$ & $A$ \\
	\midrule
	33 & 7&2&3 & 45&2&6 & 25&2&7 & \textbf{109}&\textbf{2}&\textbf{14}\\ 
	\hline
	34 & 6&0&3 & 37&0&5 & 28&0&9 & \textbf{127}&\textbf{0}&\textbf{17}\\ 
	\hline
	35 & 13&0&4 & 59&6&8 & 27&0&8 & 117&6&15\\ 
	\hline
	36 & 5&0&2 & 31&0&4 & 43&0&12 & \textbf{181}&\textbf{0}&\textbf{23}\\ 
	\hline
	37 & 43&0&6 & 94&0&12 & 45&0&7 & \textbf{98}&\textbf{0}&\textbf{13}\\ 
	\hline
	38 & 8&0&3 & 49&2&6 & 36&0&10 & \textbf{161}&\textbf{2}&\textbf{20}\\ 
	\hline
	39 & 13&0&5 & 67&6&9 & 36&0&10 & \textbf{155}&\textbf{6}&\textbf{20}\\ 
	\hline
	40 & 10&0&3 & 45&4&6 & 49&0&13 & \textbf{205}&\textbf{12}&\textbf{26}\\ 
	\hline
	41 & 54&0&8 & \textbf{117}&\textbf{0}&\textbf{15} & 57&0&8 & \textbf{123}&\textbf{0}&\textbf{16}\\ 
	\hline
	42 & \textit{1}&\textit{0}&\textit{3} & 37&2&5 & \textbf{33}&\textbf{0}&\textbf{18} & \textbf{289}&\textbf{6}&\textbf{36}\\ 
	\hline
	43 & 60&0&8 & \textbf{130}&\textbf{0}&\textbf{16} & 63&0&9 & \textbf{136}&\textbf{0}&\textbf{17}\\ 
	\hline
	44 & 13&0&4 & 63&4&8 & \textbf{55}&\textbf{2}&\textbf{14} & \textbf{229}&\textbf{8}&\textbf{28}\\ 
	\hline
	45 & 17&0&5 & 79&2&10 & 55&0&15 & \textbf{235}&\textbf{14}&\textbf{29}\\ 
	\hline
	46 & 13&0&5 & 73&3&9 & \textbf{55}&\textbf{0}&\textbf{15} & \textbf{241}&\textbf{9}&\textbf{29}\\ 
	\hline
	47 & 73&0&10 & \textbf{157}&\textbf{5}&\textbf{19} & 77&0&10 & \textbf{165}&\textbf{5}&\textbf{20}\\ 
	\hline
	48 & 9&4&4 & 57&16&7 & \textbf{81}&\textbf{4}&\textbf{21} & \textbf{337}&\textbf{20}&\textbf{41}\\ 
	\hline
	49 & 69&0&10 & \textbf{151}&\textbf{0}&\textbf{19} & \textbf{94}&\textbf{6}&\textbf{13} & \textbf{201}&\textbf{12}&\textbf{25}\\ 
	\hline
	50 & 14&0&5 & 73&0&9 & \textbf{77}&\textbf{0}&\textbf{20} & \textbf{331}&\textbf{0}&\textbf{40}\\ 
	\hline
	51 & 25&5&8 & \textbf{121}&\textbf{8}&\textbf{15} & \textbf{64}&\textbf{5}&\textbf{17} & \textbf{271}&\textbf{8}&\textbf{33}\\ 
	\hline
	52 & 21&0&6 & 93&0&11 & \textbf{78}&\textbf{0}&\textbf{20} & \textbf{323}&\textbf{0}&\textbf{39}\\ 
	\hline
	53 & \textbf{96}&\textbf{0}&\textbf{13} & \textbf{204}&\textbf{0}&\textbf{25} & \textbf{100}&\textbf{0}&\textbf{13} & \textbf{212}&\textbf{0}&\textbf{26}\\ 
	\hline
	54 & 12&0&5 & 73&0&9 & \textbf{100}&\textbf{0}&\textbf{26} & \textbf{433}&\textbf{0}&\textbf{52}\\ 
	\hline
	55 & 38&0&10 & \textbf{163}&\textbf{6}&\textbf{20} & \textbf{70}&\textbf{0}&\textbf{18} & \textbf{295}&\textbf{6}&\textbf{35}\\ 
	\hline
	56 & 21&0&6 & \textbf{101}&\textbf{4}&\textbf{12} & \textbf{97}&\textbf{4}&\textbf{24} & \textbf{401}&\textbf{12}&\textbf{48}\\ 
	\hline
	57 & 31&5&9 & \textbf{153}&\textbf{8}&\textbf{18} & \textbf{81}&\textbf{5}&\textbf{20} & \textbf{341}&\textbf{8}&\textbf{40}\\ 
	\hline
	58 & 24&0&8 & \textbf{121}&\textbf{0}&\textbf{15} & \textbf{91}&\textbf{0}&\textbf{23} & \textbf{391}&\textbf{0}&\textbf{46}\\ 
	\hline
	59 & \textbf{121}&\textbf{3}&\textbf{15} & \textbf{256}&\textbf{3}&\textbf{30} & \textbf{126}&\textbf{3}&\textbf{16} & \textbf{266}&\textbf{3}&\textbf{31}\\ 
	\hline
	60 & 7&0&5 & 73&4&9 & \textbf{79}&\textbf{0}&\textbf{38} & \textbf{649}&\textbf{12}&\textbf{76}\\ 
	\hline
	61 & \textbf{131}&\textbf{0}&\textbf{17} & \textbf{276}&\textbf{0}&\textbf{33} & \textbf{135}&\textbf{0}&\textbf{17} & \textbf{284}&\textbf{0}&\textbf{34}\\ 
	\hline
	62 & 28&0&9 & \textbf{141}&\textbf{3}&\textbf{17} & \textbf{105}&\textbf{0}&\textbf{27} & \textbf{449}&\textbf{9}&\textbf{53}\\ 
	\hline
	63 & 35&0&10 & \textbf{171}&\textbf{10}&\textbf{20} & \textbf{109}&\textbf{0}&\textbf{27} & \textbf{457}&\textbf{14}&\textbf{53}\\ 
	\hline
	64 & 70&6&9 & \textbf{155}&\textbf{14}&\textbf{18} & \textbf{225}&\textbf{30}&\textbf{27} & \textbf{465}&\textbf{62}&\textbf{54}\\ 
	\bottomrule
\end{tabular}
\end{center}
\label{tablelowlvls2}
    \end{minipage}
\end{table}

\newpage
\bibliographystyle{amsalpha}
\bibliography{bibliobranch}{}

\end{document}